\documentclass[11pt, a4paper]{amsart}

\usepackage{graphics,color,pgf,comment}
\usepackage{epsfig}

 \usepackage[ansinew]{inputenc}
 \usepackage[all]{xy}
 \SelectTips{cm}{}
 \usepackage{hyperref}

\usepackage{amsfonts}
\usepackage{amssymb, amsmath,amsthm, mathtools}
\usepackage{tabularx, hyperref}
\usepackage{enumerate}

\usepackage{tikz-cd}

\makeatletter
\newtheorem*{rep@theorem}{\rep@title}
\newcommand{\newreptheorem}[2]{%
\newenvironment{rep#1}[1]{%
 \def\rep@title{#2 \ref{##1}}%
 \begin{rep@theorem}}%
 {\end{rep@theorem}}}
\makeatother

\newtheorem{intro_thm}{Theorem}
\newtheorem{intro_cor}[intro_thm]{Corollary}

\newtheorem{lemma}{Lemma}[section]
\newtheorem{thm}[lemma]{Theorem}
\newreptheorem{thm}{Theorem}  
\newtheorem{prop}[lemma]{Proposition}
\newreptheorem{prop}{Proposition}  

\newtheorem{cor}[lemma]{Corollary}
\newreptheorem{cor}{Corollary}  

\theoremstyle{definition}
\newtheorem{defn}[lemma]{Definition}

\newtheorem{rem}[lemma]{Remark}

\newtheorem{setup}[lemma]{Setup}
\theoremstyle{definition}

\makeatletter
\newcommand\norm{\bBigg@{0.8}}

\makeatother
 
 \newcommand{\indnorm}[2][flex]{\csname #1l\endcsname\|#2%
                                 \csname #1r\endcsname\|\mathclose{}}
                                  \newcommand{\indnorml}[4][flex]{\csname #1l\endcsname\|#2%
                                 \csname #1r\endcsname\|_{#3}^{#4}\mathclose{}}
\newcommand{\sv}[2][flex]{\indnorm[#1]{#2}}
\newcommand{\isv}[2][norm]{\indnorml[#1]{#2}{\Z}{}}

\DeclareMathOperator{\res}{res}
\DeclareMathOperator{\diam}{diam}

\DeclareMathOperator{\vol}{vol}
\DeclareMathOperator{\hypvol}{hypvol}

\DeclareMathOperator{\EMD*}{EMD^*}
\DeclareMathOperator{\MD}{MD}
\DeclareMathOperator{\snap}{snap}
\DeclareMathOperator{\Lip}{Lip}
\DeclareMathOperator{\straight}{straight}
\DeclareMathOperator{\size}{size}
\DeclareMathOperator{\Aut}{Aut}

\newcommand{\pfc}[1]{%
  \widehat{#1}}
\newcommand{\fa}[1]{%
  \forall_{#1}\quad}
\newcommand{\qand}{%
  \qquad\text{and}\qquad}
\newcommand{\connsum}{\mathbin{\#}}
\newcommand{\symmdiff}{%
  \mathbin{\triangle}}
\def\ltb#1{%
  b^{(2)}_{#1}}

\def\Hyp{\mathbb{H}}

\newcommand{\N}{\ensuremath {\mathbb{N}}}
\newcommand{\R} {\ensuremath {\mathbb{R}}}

\newcommand{\Z} {\ensuremath {\mathbb{Z}}}

\newcommand{\Isom}{\ensuremath{{\rm Isom}}}
\renewcommand{\d}{\ensuremath{\,\mathrm{d}}}
\renewcommand{\rho}{\varrho}
\def\phi{\varphi}

\newcommand{\lf}{\ensuremath{\mathrm{lf}}}
\newcommand{\dvol}{\ensuremath{\mathrm{dvol}}}

\DeclareMathOperator{\map}{map}

\def\linfz#1{%
  L^\infty(#1;\Z)}

\newcommand{\ifsv}[2][norm]{\!\csname #1l\endcsname\bracevert\!#2\!%
                            \csname #1r\endcsname\bracevert\!}
\newcommand{\ifsvlf}[2][norm]{\!\csname #1l\endcsname\bracevert\!#2\!%
                            \csname #1r\endcsname\bracevert\!_{\lf}}
\newcommand{\stisv}[2][norm]{\indnorml[#1]{#2}{\Z}{\infty}}

\def\actson{\curvearrowright}

\usepackage{color}
\usepackage{pdfcolmk}

\def\longrightarrow{\rightarrow}
\def\longmapsto{\mapsto}
\def\widetilde{\tilde}

\begin{document}

\title[Stable integral simplicial volume of $3$-manifolds]
      {Stable integral simplicial volume\\ of $3$-manifolds}

\author[]{Daniel Fauser}
\address{Fakult\"{a}t f\"{u}r Mathematik, Universit\"{a}t Regensburg, 93040 Regensburg, Germany}
\email{daniel.fauser@gmx.de}

\author[]{Clara L\"{o}h}
\address{Fakult\"{a}t f\"{u}r Mathematik, Universit\"{a}t Regensburg, 93040 Regensburg, Germany}
\email{clara.loeh@ur.de}

\author[]{Marco Moraschini}
\address{Fakult\"{a}t f\"{u}r Mathematik, Universit\"{a}t Regensburg, 93040 Regensburg, Germany}
\email{marco.moraschini@ur.de}

\author[]{Jos\'e Pedro Quintanilha}
\address{Fakult\"{a}t f\"{u}r Mathematik, Universit\"{a}t Regensburg, 93040 Regensburg, Germany}
\email{jose-pedro.quintanilha@ur.de}

\thanks{}

\date{\today.\ \copyright{\ D.~Fauser, C.~L\"oh, M.~Moraschini, J.~P.~Quintanilha 2019}.
  This work was supported by the CRC~1085 \emph{Higher Invariants}
  (Universit\"at Regensburg, funded by the DFG)}

\begin{abstract}
 We show that non-elliptic prime $3$-manifolds 
 satisfy integral approximation for the simplicial volume, i.e., that
 their simplicial volume equals the stable integral simplicial volume. 
 The proof makes use of integral foliated simplicial volume and tools
 from ergodic theory.  
\end{abstract}

\maketitle

\section{Introduction}

The \emph{simplicial volume} of an oriented compact $n$-manifold~$M$
(possibly with non-empty boundary) over a normed ring~$R$ is defined by
\begin{align*}
  \sv{M,\partial M}_R
   \coloneqq \inf \biggl\{ \sum_{j=1}^m |a_j|
   \biggm| & \; \sum_{j=1}^m a_j \cdot \sigma_j
   \in C_n(M;R) \text{ is a relative}\\[-1em]
   & \; \text{$R$-fundamental cycle of~$(M,\partial M)$}\biggr\},
\end{align*}
which is an algebraic version of (stable) complexity of manifolds. 
The classical case is~$\sv{M,\partial M} \coloneqq \sv{M,\partial M}_\R$,
introduced by Gromov~\cite{GrMo,Grom82,mapsimvol} in the context of
hyperbolic geometry and the study of topological properties of (minimal)
volume.

\subsection{The approximation problem for simplicial volume}

If $M$ admits enough finite coverings (i.e., if $\pi_1(M)$ is residually finite),
it makes sense to consider the \emph{stable integral simplicial volume}
\[ \stisv{M,\partial M}
   \coloneqq \inf \Bigl\{ \frac{\isv{W,\partial W}}{d}
           \Bigm| \text{$d \in \N$,\ $W$ a $d$-sheeted covering of~$M$}
	   \Bigr\}.
\]
In the closed case, stable integral simplicial volume gives an upper
bound for $L^2$-Betti numbers~\cite[p.~305]{Gromovbook}\cite{Sthesis},
logarithmic torsion growth of
homology~\cite[Theorem~1.6]{FLPS}\cite{sauergrowth}, and the rank
gradient~\cite{loehrg}.

As for Betti numbers, ranks of fundamental groups, or logarithmic
torsion of homology, one can ask which (typically aspherical)
manifolds~$M$ satisfy integral approximation for simplicial volume,
i.e.,~$\sv{M,\partial M} = \stisv{M,\partial M}$.

The main goal of this paper is to show that non-elliptic prime
$3$-manifolds satisfy integral approximation for simplicial volume
(Theorem~\ref{thm:main}) and that reducible $3$-manifolds in general
do not (Section~\ref{subsec:non-approxintro}), thereby answering the
approximation question in the $3$-dimensional
case~\cite[Question~1.3]{FFL}.

The following classes of manifolds were already known to satisfy
integral approximation for simplicial volume: closed surfaces of
positive genus~\cite[p.~9]{Grom82}, closed hyperbolic
$3$-manifolds~\cite[Theorem~1.7]{FLPS}, closed aspherical manifolds
with amenable residually finite fundamental
group~\cite[Theorem~1.10]{FLPS}, compact manifolds where $S^1$~acts
``non-trivially''~\cite{fauserT,fauserS1}, as well as graph manifolds
not covered by~$S^3$~\cite{FFL}.

In contrast, approximation fails uniformly for higher-dimensional
hyperbolic manifolds~\cite[Theorem~2.1]{FFM}
and it fails for closed manifolds with non-abelian free
fundamental group~\cite[Remark~3.9]{FLPS}.

\subsection{Main approximation result}

More precisely, we have the following positive result, 
which includes all closed aspherical $3$-manifolds.

\begin{intro_thm}[integral approximation for simplicial volume of non-elliptic prime $3$-manifolds]
\label{thm:main}
  Let $M$ be an oriented compact connected $3$-manifold with empty or
  toroidal boundary. If $M$~is prime and not covered by~$S^3$, then
  \[   \stisv {M,\partial M}
     = \sv {M, \partial M}
     = \frac {\hypvol(M)}{v_3}.
  \]
\end{intro_thm}

Here, $v_3$ is the volume of a (whence every) ideal regular
tetrahedron in~$\Hyp^3$, and $\hypvol(M)$~is the total volume of the
hyperbolic pieces in the JSJ decomposition of~$M$ (see
Definition~\ref{def:hypvol}).  The equality~$\sv{M,\partial M} =
\hypvol(M) / v_3$ follows from the work of Soma~\cite{Soma}, which
also holds in the non-prime case.

\subsection{Non-approximation results}\label{subsec:non-approxintro}

In the non-prime case, not all closed $3$-manifolds (with infinite
fundamental group) satisfy integral approximation for simplicial
volume; similar to previously known non-approximation results via the
first $L^2$-Betti number~\cite[Remark~3.9]{FLPS}, we obtain (see
Section~\ref{sec:nonapprox} for the proofs):

\begin{intro_thm}\label{thm-noapprox}
  Let $d \in \N_{\geq 3}$, let $m,n \in \N$, let $M_1,\dots, M_m$, $N_1,\dots, N_n$
  be oriented closed connected $d$-manifolds with the following properties:
  \begin{enumerate}
  \item We have $\sv {M_j} > 0$  for all~$j \in
    \{1,\dots, m\}$ as well as $\sv {N_k} = 0$ for all~$k \in \{1,\dots, n\}$.
  \item Moreover, $m + n - 1 - \sum_{k=1}^n 1 / |\pi_1(N_k)| > \sum_{j=1}^m \sv {M_j}$ (with
    the convention that $1/\infty := 0$).
  \end{enumerate}
  Then the connected sum~$M \coloneqq M_1 \connsum \dots \connsum M_m
  \connsum N_1 \connsum \dots \connsum N_n$ does \emph{not} satisfy
  integral approximation for simplicial volume, i.e., we have~$\sv M <
  \stisv M$.
\end{intro_thm}

\begin{intro_cor}\label{cor:intro:nonapprox}
  Let $N$ be an oriented closed connected hyperbolic $3$-manifold and
  let $k > \vol(N)/v_3$. Then the oriented closed connected
  $3$-manifold~$M \coloneqq N \connsum \connsum^k (S^1)^3$
  satisfies~$\sv M < \stisv M$.
\end{intro_cor}

In the case of closed $3$-manifolds with vanishing simplicial volume,
we have a complete characterisation of approximability:

\begin{intro_cor}\label{cor:0noapprox}
  Let $M$ be an oriented closed connected $3$-manifold with $\sv M = 0$. Then
  the following are equivalent:
  \begin{enumerate}
  \item The simplicial volume of~$M$ satisfies integral approximation, i.e., $\stisv M = \sv M$.
  \item The manifold~$M$ is prime and has infinite fundamental group or $M$ is
    homeomorphic to~$\R P^3 \connsum \R P^3$.
  \end{enumerate}       
\end{intro_cor}

\subsection{Strategy of proof of Theorem~\ref{thm:main}}\label{sec:intro:soma}

Clearly, we have~$\stisv{M,\partial M} \geq \sv{M,\partial M}$
(by the degree estimate~\cite[p.~8]{Grom82}).
Therefore, in combination with Soma's computation~\cite{Soma},
in the situation of Theorem~\ref{thm:main}, we have
\[ \stisv{M,\partial M}
   \geq \sv{M,\partial M}
   =    \frac{\hypvol(M)}{v_3}.
\]
Thus, it suffices to show the converse estimate
\[ \stisv{M,\partial M} \leq \frac{\hypvol(M)}{v_3}.
\]
As in Soma's computation of the classical simplicial volume of
$3$-manifolds, we use the JSJ decomposition and hyperbolisation to cut
irreducible manifolds~$M$ along tori into pieces~$W$ that are
hyperbolic or Seifert-fibered, and not covered by~$S^3$ (the
additional case~$M\cong S^1 \times S^2$ being also Seifert-fibered).

If a piece $W$ is Seifert-fibered (or, more generally, a non-elliptic 
graph manifold), then it is known that
$\stisv {W,\partial W} = 0 = \hypvol(W)/v_3$~\cite[Section~8]{LP}\cite{FFL}.

Therefore, two main challenges remain:
\begin{itemize}
\item
  the hyperbolic case with toroidal boundary and
\item
  subadditivity with respect to glueings along tori.
\end{itemize}
To this end, it is convenient to rewrite stable integral
simplicial volume as parametrised simplicial volume
with respect to the canonical action on the profinite
completion of the fundamental group(oid)
(Section~\ref{subsec:profinsimvol}):
\[ \stisv {W,\partial W} = \ifsv {W,\partial W}^{\pfc{\pi_1(W)}}.
\]

For subadditivity with respect to glueings along tori,
we need control over the size of the boundaries of the
relative fundamental cycles. In order to avoid a 
technically demanding equivalence theorem, we proceed
similarly to Soma's work with parametrised relative simplicial
volume~$\ifsv{W,\partial W}_\partial^{\pfc{\pi_1(W)}}$ with boundary control
(Section~\ref{subsec:boundarycontrolsimvol}). Then, as
in the case of graph manifolds~\cite{FFL}, we can use the uniform
boundary condition on tori in the para\-metrised setting~\cite{fauserloehUBC},
to establish subadditivity (Section~\ref{sec:glue}). A
subtle point is that, during the glueing step, we also
need to stay in control of the parameter spaces; at
this point, we will use profinite properties of
JSJ decompositions (Section~\ref{subsec:profin3}).

Finally, we need to show for hyperbolic pieces~$W$ that
$\ifsv{W,\partial W}_\partial^{\pfc{\pi_1(W)}} = \vol(W^\circ)/v_3$.
This generalisation of the closed case~\cite[Theorem~1.7]{FLPS}
will take up a large part of the paper. More specifically, we
will proceed in two steps:

First, using a suitably
adapted smearing process and the approximation results
in the closed case, we establish the following
proportionality in the open case (Section~\ref{sec:hyplf}):

\begin{intro_thm}\label{thm:hyp3}
  Let $M$ be an oriented complete connected finite-volume hyperbolic
  $3$-manifold (without boundary). Then
  \[ \ifsvlf M 
   = \frac{\vol(M)}{v_3}.
  \]
\end{intro_thm}

Second, we relate this locally finite version to the parametrised
simplicial volume with boundary control of the ambient compact
manifold~$W$: We have~$\ifsvlf M \geq \ifsv{W,\partial W}_\partial$
(Section~\ref{subsec:gaincontrol}).  Combining the fact that
$\pi_1(W)$~satisfies Property~$\EMD*$ from ergodic theory
(Proposition~\ref{prop:emd}) with monotonicity of boundary-controlled
integral foliated simplicial volume with respect to weak containment
of parameter spaces (Proposition~\ref{prop:weakcont}), we then obtain
(Section~\ref{sec:proof:cor:prof}):

\begin{intro_cor}\label{cor:hyp3rel}
  Let $W$ be an oriented compact connected hyperbolic $3$-mani\-fold
  with empty or toroidal boundary and let $M \coloneqq W^\circ$. Then 
  \[ \ifsv{W,\partial W}^{\widehat{\pi_1(W)}}_\partial
   = \frac{\vol(M)}{v_3}.
  \]
\end{intro_cor}

\subsection*{Overview of this article}

We recall basic terminology related to integral foliated/parametrised
simplicial volume in Section~\ref{sec:simvolcompact}.  In
Section~\ref{sec:simvolopen}, we generalise this setup to the case of
open manifolds. We then prove proportionality for complete hyperbolic
$3$-manifolds of finite volume: The locally finite case is established
in Section~\ref{sec:hyplf}; the relative case is derived in
Section~\ref{sec:hyprel}. The JSJ glueing argument and the proof of
the main theorem (Theorem~\ref{thm:main}) are explained in
Section~\ref{sec:glue}.  Finally, in Section~\ref{sec:nonapprox}, we
prove the non-approximation results.

\subsection*{Acknowledgements}

We are grateful to the anonymous referee for the many useful comments and suggestions
to improve this paper. 

\section{Integral foliated simplicial volume: The compact case}\label{sec:simvolcompact}

We recall basic terminology related to integral foliated simplicial
volume in the compact case; the open case will be covered in
Section~\ref{sec:simvolopen}. Integral foliated simplicial volume of a
compact manifold~$M$ is a variation of simplicial volume with local
coefficients in integer-valued $L^\infty$-functions on a probability
space with a $\pi_1(M)$-action.  Therefore, we first briefly review
normed local coefficients.

\subsection{Normed local coefficients}

We now focus on local coefficient systems that carry a compatible
norm. Our main example will be spaces of essentially bounded functions
(with the $L^1$-norm) on standard Borel probability spaces.

\begin{defn}[normed local coefficient system]
	Let~$W$ be a topological space. A \emph{normed local
          coefficient system on~$W$} is a functor from the fundamental
        groupoid~$\Pi(W)$ of~$W$ to the category of normed abelian
        groups and norm non-increasing group homomorphisms.
\end{defn}

Given a $k$-simplex $\sigma \in \map(\Delta^k, W)$ in a topological
space~$W$, the $l$-dimensional face of $\sigma$ spanned by the
vertices $\sigma(e_{i_0}), \ldots, \sigma(e_{i_l})$ (with $0 \le i_0 <
\ldots < i_l \le k$) will be denoted by~$\sigma[i_0, \ldots, i_l]$.

\begin{defn}[chain complex with local coefficients]
\label{def:chainslc}
	Let~$W$ be a topological space and let~$L$ be a normed local
        coefficient system on~$W$.  Then the \emph{chain complex
          of~$W$ with local coefficients in~$L$} is given by
	\[ C_k(W;L) \coloneqq \bigoplus_{\sigma\in \map(\Delta^k,W)} L\bigl(\sigma[0]\bigr)
	\]
	for each $k\in \N$, with boundary operators
	\begin{align*}
		C_k(W;L) &\longrightarrow C_{k-1}(W;L)\\
		a\cdot \sigma &\longmapsto L\bigl(\sigma[0,1]\bigr) (a) \cdot \partial_0 \sigma
				+ \sum_{i=1}^k {(-1)^i \cdot a \cdot \partial_i \sigma}.
	\end{align*}
	We equip it with the $\ell^1$-norm~$|\cdot |_{1,L}$ induced by
        the norm on~$L$.

	If~$V\subseteq W$ is a subspace, we write~$C_*^W(V;L)$ for all
        chains in $C_*(W;L)$ that are supported in~$V$.  We define the
        \emph{chain complex of~$W$ relative to~$V$ with local
          coefficient system~$L$} by
        \[ C_*(W,V;L) \coloneqq C_*(W;L) / C_*^W(V;L)
        \]
        and endow it with the quotient norm of~$|\cdot|_{1,L}$.
\end{defn}

\begin{defn}[homology with local coefficients]
  Let~$W$ be a topological space, let~$L$ be a normed local
  coefficient system on~$W$ and let~$V\subseteq W$ be a subspace.  We
  define the $k$-th \emph{homology group with local coefficient
    system~$L$} by
  \[ H_k(W,V;L) \coloneqq H_k\bigl( C_*(W,V;L)\bigr)
  \]
  and write $\|\cdot\|_{1,L}$ for the induced $\ell^1$-seminorm on
  homology.
\end{defn}

\begin{defn} [standard $G$-space]
	Let~$G$ be a groupoid. A \emph{standard $G$-space} is a contravariant functor from~$G$ to 
	the category of all standard Borel probability spaces and probability
	measure-preserving transformations.
\end{defn}

\begin{defn}[associated normed local coefficient system]\label{def:ass:norm:loc:coeff:syst}
	Let~$G$ be a groupoid and let~$\alpha$ be a standard
        $G$-space.  Then the \emph{associated normed local coefficient
          system}~$\linfz \alpha$ \emph{to}~$\alpha$ \emph{on}~$G$ is
        the post-composition of~$\alpha$ with the (contravariant)
        dualising functor~$L^\infty(-,\Z)$. In other words,
	\[ \linfz \alpha (x) \coloneqq \linfz {\alpha(x)}
	\]
	for all objects~$x$ in~$G$ (equipped with the $L^1$-norm), and
	\begin{align*}
	\linfz \alpha (g) \colon \linfz {\alpha(x)} &\longrightarrow \linfz {\alpha (y)} \\
	f &\longmapsto f\circ \alpha (g)
	\end{align*}
	for all morphisms~$g\colon x \longrightarrow y$ in $G$.
\end{defn}

\begin{rem}[from groups to groupoids] \label{rem:groupoidvsgroup}
	Local coefficient systems are quite similar to the more
        conventional twisted coefficients.  In the setting of local
        (rather than twisted) coefficients, the role of the
        fundamental group is played by the fundamental groupoid, which
        is convenient because it spares us from caring about
        basepoints, and from working at the level of the universal
        cover.
	
	Given a topological space~$W$ and a standard
        $\Pi(W)$-space~$\alpha$, we obtain, for each choice of
        basepoint~$x_0\in W$, a canonical standard $\pi_1(W,
        x_0)$-space by restriction of~$\alpha$ to~$x_0$.  That is, we
        let $\pi_1(W, x_0)$~act on the Borel probability
        space~$\alpha(x_0)$ by $\gamma\cdot x\coloneqq
        \alpha(\gamma)(x)$, where $\gamma\in\pi_1(W, x_0)$ and $x\in
        \alpha(x_0)$.
	
	Conversely, if $W$~is path-connected, then each standard
        $\pi_1(W, x_0)$-space~$X$ can be extended to a standard
        $\Pi(W)$-space~$\alpha$ by first choosing, for each
        point~$p\in W$, a path~$\gamma_p$ (up to homotopy class
        relative to endpoints) from$~x_0$ to~$p$, and then setting:
	\begin{itemize}
		\item At every point~$p \in W$, we put~$\alpha(p) \coloneqq X$.
		\item For each morphism~$(\gamma \colon p \rightarrow q)$ in~$\Pi(W)$, we
                  set~$h_\gamma := \gamma_p * \gamma * \gamma_q^{-1} \in \pi_1(W,x_0)$ and 
		\begin{align*}
		\alpha(\gamma) \colon X & \to X \\
		x &\mapsto h_\gamma \cdot x.
		\end{align*}
	\end{itemize}
	It is easily verified that this makes~$\alpha$ a contravariant
        functor, whose restriction to~$x_0$, as explained previously,
        recovers the left $\pi_1(W, x_0)$-action on~$X$. This is
        ultimately a reflection of the fact that when
        $\pi_1(W,x_0)$~is regarded as a one-object sub-category
        of~$\Pi(W)$, the inclusion~$\pi_1(W,x_0) \hookrightarrow
        \Pi(W)$ is an equivalence of categories.
	
	We should also mention that the choice of path classes
        $(\gamma_p\colon x_0 \to p)_{p \in W}$ is immaterial, as
        picking a different collection $(\gamma'_p\colon x_0 \to p)_{p
          \in W}$ leads to the construction of a contravariant
        functor~$\alpha'$ that is isomorphic to~$\alpha$, in the
        category-theoretical sense.  Indeed, it is easy to verify that
        the maps
	\begin{align*}
	X & \to X\\
	x & \mapsto (\gamma_p' *\gamma_p^{-1}) \cdot x,
	\end{align*}
	over all~$p\in W$, assemble to a natural
        transformation~$\alpha \to \alpha'$, which is clearly
        invertible.
	
	Similarly, a normed local coefficient system~$L$ on a
        topological space~$W$ can be restricted to a chosen
        basepoint~$x_0 \in W$, yielding a normed \emph{right}
        $\pi_1(W, x_0)$-module~$L(x_0)$. And conversely, a right
        $\pi_1(W, x_0)$-action on a normed abelian group~$A$ can be
        extended to a normed local coefficient system~$L$ that is
        constantly~$A$ on objects, by choosing paths $\gamma_p\colon
        x_0\to p$ as before, and setting, for each~$(\gamma\colon p\to
        q)$ in~$\Pi(W)$,
	\begin{align*}
	  L(\gamma) \colon A & \to A \\
	  a &\mapsto a \cdot (\gamma_p * \gamma *\gamma_q^{-1}).
	\end{align*}
	It is straightforward to check that homology with twisted
        coefficients in a normed right $\pi_1(W,x_0)$-module $A$ is
        isomorphic to homology with local coefficients in any normed
        local coefficient system obtained as an extension of~$A$.
	
	We also remark that all these constructions are compatible
        with the dualising procedure introduced in
        Definition~\ref{def:ass:norm:loc:coeff:syst}.  Indeed, the
        construction of the associated normed local coefficient system
        to a standard $\Pi(W)$-space fits into the following
        commutative diagram:\medskip
	
	\centering
	\begin{tikzcd}
	\begin{minipage}{0.25\textwidth}
	\centering
	Standard (left) $\pi_1(W,x_0)$-spaces
	\end{minipage} \arrow[d, "{L^\infty(-,\Z)}"] \arrow[r, "\text{extend}"]&
	\begin{minipage}{0.25\textwidth}
	\centering
	Standard $\Pi(W)$-spaces
	\end{minipage} \arrow[d, "{L^\infty(-,\Z)}"] \arrow[r, "\text{restrict}"]  &
	\begin{minipage}{0.25\textwidth}
	\centering
	Standard (left) $\pi_1(W,x_0)$-spaces
	\end{minipage} \arrow[d, "{L^\infty(-,\Z)}"]\\
	\begin{minipage}{0.25\textwidth}
	\centering
	Normed right $\pi_1(W, x_0)$-modules
	\end{minipage} \arrow[r, "\text{extend}"]&
	\begin{minipage}{0.25\textwidth}
	\centering
	Normed local coefficient systems
	\end{minipage} \arrow[r,"\text{restrict}"]&\begin{minipage}{0.25\textwidth}
	\centering
	Normed right $\pi_1(W, x_0)$-modules
	\end{minipage}
	\end{tikzcd}
	
\end{rem}

\subsection{Integral foliated simplicial volume}

Integral foliated simplicial volume relaxes the integrality of coefficients
by allowing for integer-valued coefficient functions on a probability space.

\begin{defn}[parametrised fundamental class with local coefficients]
	Let~$M$ be an oriented compact connected $n$-manifold and
        let~$\alpha$ be a standard $\Pi(M)$-space.  Then the
        \emph{$\alpha$-\emph{parametrised fundamental
            class}~$[M,\partial M]^\alpha$ of~$M$} is defined to be
        the image of the integral fundamental class~$[M,\partial M]$
        of~$M$ under the change of coefficient map
	\[ H_n(M,\partial M;\Z) \longrightarrow H_n\bigl(M,\partial M;\linfz \alpha\bigr) 
	\]
	induced by the inclusion of~$\Z$ into~$\linfz \alpha$ as constant functions.
\end{defn}

\begin{defn}[integral foliated simplicial volume]
  Let~$M$ be an oriented compact connected manifold and
  let $\alpha$ be a standard $\Pi(M)$-space. Then the
  \emph{$\alpha$-parametrised simplicial volume of~$M$} is
  defined by 
  \[ \ifsv{M,\partial M}^\alpha \coloneqq \sv{[M,\partial M]^\alpha}_{1,\linfz \alpha}.
  \]
  Moreover, the \emph{integral foliated simplicial volume~$\ifsv
    {M,\partial M}$ of~$M$} is the infimum of all parametrised
  simplicial volumes, where the infimum is taken over all standard
  $\Pi(M)$-spaces.

  If $M$ is closed, we also write~$\ifsv M^\alpha$ and $\ifsv M$
  for the corresponding quantities.
\end{defn}
	
\begin{prop}[comparison with ordinary simplicial volume~\protect{\cite[Theorem~5.35]{Sthesis}}\cite{Gromovbook}]\label{prop:boundaryreal}
  Let $W$ be an oriented compact connected manifold and let $\alpha$
  be a standard $\Pi(M)$-space. Then integration of the coefficients
  shows that
  \[ \sv{W,\partial W} \leq \ifsv{W,\partial W}^\alpha.
  \]
  In particular, $\sv{W,\partial W} \leq \ifsv{W,\partial W}$.
\end{prop}

\subsection{The profinite completion}\label{subsec:profinsimvol}

The link between stable integral simplicial volume and the
parametrised simplicial volume and ergodic theory
is given by the profinite completion of the fundamental group.

\begin{defn}[profinite completion of a group]\label{def:pfc}
	Let $\Gamma$ be a countable group and consider the inverse
        system of finite-index normal subgroups of~$\Gamma$ together
        with their inclusion homomorphisms. The inverse limit
	\[ \pfc\Gamma \coloneqq \varprojlim_{\Lambda \underset{\text{f. i.}}{\trianglelefteq} \Gamma} \Gamma/\Lambda\]
	of the corresponding group quotients is called the
        \emph{profinite completion} of~$\Gamma$.  The left translation
        action and the normalised counting measures on the
        quotients~$\Gamma/\Lambda$ turn~$\pfc\Gamma$ into a standard
        $\Gamma$-space.

	Moreover, group homomorphisms induce canonical maps on
        profinite completions, making this definition functorial
        \cite[Lemma~3.2.3]{ribes2000profinite}.
\end{defn}

If $\Gamma$ is a countable group, then a straightforward computation
shows that the standard $\Gamma$-space~$\pfc \Gamma$ is (essentially)
free if and only if $\Gamma$ is residually finite.

\begin{defn}[profinite completion of the fundamental groupoid]
  Let $W$ be a path-connected topological space. The standard
  $\Pi(W)$-space~$\pfc{\Pi(W)}$ is defined as follows:
  \begin{itemize}
  \item For each~$p \in W$, we set~$\pfc{\Pi(W)}(p) \coloneqq \pfc{\pi_1(W,p)}$.
  \item Given a morphism~$\gamma \colon p \rightarrow q$ in~$\Pi(W)$, we take
  $\pfc{\Pi(W)}(\gamma) \colon \pfc{\pi_1(W,q)} \to \pfc{\pi_1(W,p)}$
  to be the map induced on profinite completions by  
    \begin{align*}
	\pi_1(W, q) &\to \pi_1(W, p)\\
	\alpha & \mapsto \gamma * \alpha * \gamma^{-1}.
    \end{align*}
  \end{itemize}
\end{defn}
	
If $x_0 \in W$, then the standard $\Pi(W)$-space~$\pfc{\Pi(W)}$ is
isomorphic to the ``extension'' construction from
Remark~\ref{rem:groupoidvsgroup}, applied to the standard
$\pi_1(W,x_0)$-space~$\pfc{\pi_1(W,x_0)}$.

\begin{prop}[stable integral simplicial volume via profinite completion]\label{prop:ifsvprofin}
  Let $W$ be an oriented compact connected manifold.
  Then
  \[ \ifsv{W,\partial W}^{\pfc{\Pi(W)}}
     = \stisv{W,\partial W}.
  \]
\end{prop}
\begin{proof}
  The proof for the closed case (with twisted
  coefficients)~\cite[Remark~6.7]{LP}\cite[Theorem~2.6]{FLPS} can be
  adapted to the relative case (with local coefficients) in a
  straightforward manner.
\end{proof}

\subsection{Adding boundary control}\label{subsec:boundarycontrolsimvol}

When proving additivity estimates for simplicial volumes under glueings,
one needs additional control on the boundary. We will use the following
version of integral foliated simplicial volume,
similar to the relative simplicial volume considered
by Thurston~\cite[Chapter~6.5]{Thurston}:

\begin{defn}[integral foliated simplicial volume with boundary control]
  Let $W$ be an oriented compact connected manifold and let $\alpha$
  be a standard $\Pi(W)$-space. Then the \emph{controlled
    $\alpha$-parametrised simplicial volume of~$W$} is defined by
  \begin{align*}
    \ifsv {W,\partial W}_\partial^\alpha
    \coloneqq \sup_{\varepsilon \in \R_{>0}}
    \inf
    \bigl\{ |c|_1
    \bigm|
    & \; \text{$c \in C_n(W; \linfz \alpha)$ relative $\alpha$-fundamental}
    \\
    & \; \text{cycle of~$W$,\
          $|\partial c|_1 \leq \varepsilon$}
    \bigr\}
  \end{align*}
  (with the convention that $\inf \emptyset = + \infty$).
  
  Taking the infimum over all~$\alpha$, we obtain~$\ifsv {W,\partial W}_\partial$.
\end{defn}

It follows from this definition that if $\ifsv{\partial W} > 0$, then
$\ifsv{W,\partial W}_\partial = + \infty$, and we also clearly
have~$\ifsv {W,\partial W} \leq \ifsv{W,\partial W}_\partial$. Whether
the converse inequality holds is a more subtle question, because
it involves a simultaneous optimisation problem. We show in
Corollary~\ref{cor:all:simpl:vol:equal} that if $W$ is a compact
hyperbolic $3$-manifold with toroidal boundary, then this is the case.
We do have that vanishing of~$\ifsv{W, \partial W}$ implies vanishing
of~$\ifsv{W, \partial W}_\partial$:

\begin{lemma}[vanishing of integral foliated simplicial volume transfers to boundary-controlled setting]\label{lem:boundaryvanishing}
  Let $W$~be an oriented compact connected $n$-manifold and $\alpha$~a
  standard $\Pi(W)$-space.  If $\ifsv{W, \partial W}^\alpha = 0$, then
  also $\ifsv{W, \partial W}_\partial^\alpha = 0$.
\end{lemma}
\begin{proof}
  Let $\varepsilon, \varepsilon_\partial >0$.  Since $\ifsv{W,
    \partial W}^\alpha = 0$, there exists a relative
  $\alpha$-fundamental cycle~$c \in C_n(W; \linfz \alpha)$ of~$W$
  with~$|c|_1 < \min(\varepsilon, \frac{\varepsilon_\partial}{n+1})$.
  It follows that $|c|_1 < \varepsilon$ and $|\partial c|_1 \le (n+1)
  |c|_1 < \varepsilon_\partial$, giving the desired boundary control.
\end{proof}

In certain situations, it is helpful to restrict to ergodic parameter
spaces.  Here, a standard $\Pi(W)$-space $\alpha$ is \emph{ergodic} if
the $\pi_1(W,x)$-space~$\alpha(x)$ is ergodic in the classical sense
for one (hence every)~$x \in W$.

\begin{prop}[ergodic parameters suffice]\label{prop:ergodicsuff}
  Let $W$ be an oriented compact connected manifold. Then for
  each~$\varepsilon\in \R_{>0}$, there exists an ergodic standard
  $\Pi(W)$-space~$\alpha$ with
  \[ \ifsv{W,\partial W}^\alpha_\partial \leq \ifsv{W,\partial W}_\partial + \varepsilon.
  \]
\end{prop}

\begin{proof}
  As in the closed case (with twisted
  coefficients)~\cite[Proposition~4.17]{LP}, this can be shown via an
  ergodic decomposition argument -- we only need to take the
  coefficients of the boundary contribution into account.
\end{proof}

\section{Integral foliated simplicial volume: The non-compact case}\label{sec:simvolopen}

We now extend the definition of integral foliated simplicial volume
to the non-compact case.

\subsection{Basic definitions}

In the non-compact case, we will replace singular chains by
locally finite singular chains (while keeping normed local
coefficients). 

\begin{defn}[locally finite chains, homology with local coefficients]
	Let $W$ be a topological space and let~$L$ be a normed local
        coefficient system on~$W$.  A (possibly infinite)
        chain~$\sum_{\sigma\in \map(\Delta^n,W)} a_\sigma \cdot
        \sigma$ with~$a_\sigma \in L\bigl(\sigma[0]\bigr)$ is called
        \emph{locally finite} if every compact subset of~$W$
        intersects only finitely many singular simplices~$\sigma$
        in~$W$ with~$a_\sigma \neq 0$.  We write~$C_*^{\lf}(W;L)$ for
        the chain modules of all locally finite chains with local
        coefficients in~$L$ and extend the boundary operator from
        Definition~\ref{def:chainslc} to locally finite chains.
	
	Let~$V\subset W$ be a subspace.  We write~$C_*^{\lf,W}(V;L)$
        for the subcomplex of all chains in~$C_*^{\lf}(W;L)$ that are supported
        in~$V$ and define
	\[ C_*^{\lf}(W,V;L) \coloneqq C_*^{\lf}(W;L) / C_*^{\lf,W}(V;L).
	\]
	We obtain a chain complex and write~$|\cdot |_{1,L}$ for the
        $\ell^1$-norm induced by~$L$.

        Moreover, we define the $k$-th \emph{locally finite homology
          group with local coefficient system $L$} by
	\[ H_k^{\lf}(W,V;L) \coloneqq H_k \bigl( C_*^{\lf} (W,V;L) \bigr)
	\]
        for the corresponding homology and $\|\cdot\|_{1,L}$ for
        the (potentially infinite) seminorm induced on
        homology by~$|\cdot|_{1,L}$.
\end{defn}

Strictly speaking, locally finite chains are functions on the
space of singular simplices; however, the suggestive notation
as ``formal sums'' has proved to be efficient in the classical
case of locally finite homology.
It should be noted that the boundary operator is indeed well-defined
(the local finiteness condition takes care of this).

Notice that if $L$ is a functor which associates to any object of the
fundamental groupoid $\Pi(W)$ the normed abelian group~$\mathbb{R}$
and to any morphism the identity on $\mathbb{R}$, we recover the
classical definition of locally finite homology with real
coefficients. Then, one can define the locally finite simplicial
volume of an oriented connected $n$-manifold $M$ without boundary,
denoted by $\lVert M \rVert_\lf$, as the $\ell^1$-seminorm of its real
locally finite fundamental class. This is thoroughly discussed in the
literature~\cite{Grom82, Loeh, FrMo}.

\begin{defn}[locally finite fundamental class with local coefficients]
	Let~$M$ be an oriented connected $n$-manifold without boundary and
        let~$\alpha$ be a standard $\Pi(M)$-space.  Then the
        \emph{$\alpha$-parametrised locally finite fundamental
          class~$[M]^\alpha_{\lf}$ of~$M$} is defined to be the image
        of the integral locally finite fundamental class~$[M]_{\lf}$
        of~$M$ under the change of coefficients map
	\[ H_n^{\lf}(M;\Z) \longrightarrow H_n^{\lf}(M;\linfz \alpha) 
	\]
	induced by the inclusion of~$\Z$ into~$\linfz \alpha$ as
        constant functions.
\end{defn}

\begin{defn}[locally finite integral foliated simplicial volume]
	Let~$M$ be an oriented connected manifold without  boundary and
        let~$\alpha$ be a standard $\Pi(M)$-space.  Then the
        \emph{$\alpha$-parametrised locally finite simplicial volume
          of~$M$} is given by
	\[ \ifsv{M}^\alpha _{\lf} \coloneqq \sv{[M]_{\lf}^\alpha}_{1,\linfz \alpha}.
	\]
	The \emph{locally finite integral foliated simplicial
          volume~$\ifsv{M}_{\lf}$ of~$M$} is defined to be the infimum
        over all parametrised locally finite simplicial volumes
        of~$M$.
\end{defn}

If $M$ is an oriented closed connected manifold, then $\ifsv M ^\alpha
= \ifsv M ^\alpha_{\lf}$ for all standard $\Pi(M)$-spaces~$\alpha$, because
every locally finite chain on a compact space is an ordinary chain (and
vice versa).

\subsection{Integration}

Integrating parametrised locally finite chains over their coefficients
leads to the following comparison between parametrised locally finite
simplicial volume and ordinary locally finite simplicial volume.

\begin{prop}[integration of coefficients]\label{prop:integration:coeff:inequality:ord:foliated}
  Let $M$ be an oriented connected manifold without boundary and let
  $\alpha$ be a standard $\Pi(M)$-space. Then
  \begin{align*}
    I_\R \colon 
    C_*^\lf(M;\linfz \alpha) & \longrightarrow C_*^\lf(M;\R) \\
    \sum_{\sigma\in \map(\Delta^*,M)} f_\sigma \cdot \sigma
    & \longmapsto \sum_{\sigma\in \map(\Delta^*,M)} \biggl(\int_{\alpha(\sigma[0])} f_\sigma\biggr) \cdot \sigma
  \end{align*}
  is a well-defined chain map that maps $\alpha$-parametrised
  fundamental cycles to locally finite $\R$-fundamental cycles.  In
  particular,
  \[ \sv M _\lf \leq \ifsv M ^\alpha_\lf
  \]
  and so
  $$
  \sv M _\lf \leq \ifsv M _\lf.
  $$
\end{prop}
\begin{proof}
  This can be proved in the same way as in the closed
  case~\cite[Remark~5.23]{Sthesis}\cite[Proposition~4.6]{LP}.
\end{proof}

Integrating both the coefficients and the simplices (over the volume form)
provides a useful criterion to detect fundamental cycles. Integration over
simplices requires some regularity on the simplices (we will use smoothness)
as well as global bounds to ensure convergence (we will use a global Lipschitz
bound). 

\begin{defn}[integration of a simplex]
  Let~$M$ be an oriented Riemannian $n$-manifold. For a smooth singular 
  $n$-simplex~$\sigma\colon \Delta^n \to M$, we define the
  \emph{integration} of~$\sigma$ over~$M$ to be
  \[\langle \dvol_M,\sigma \rangle \coloneqq \int_{\Delta^n} \sigma^* \dvol_M.\]
\end{defn}

\begin{defn}[Lipschitz chain]
  Let $M$~be a Riemannian manifold, let $\alpha$ be a standard
  $\Pi(M)$-space and let $c\in C_*^\lf(M, \linfz \alpha)$~be an
  $\alpha$-parametrised locally finite chain. Let $\Lip(c)\in [0,
    \infty]$ denote the supremum of the (possibly infinite) Lipschitz
  constants over the simplices in~$c$ (that occur with non-zero
  coefficient). We say that $c$~is \emph{Lipschitz} if $\Lip(c) <
  \infty$.
\end{defn}

An important class of Lipschitz simplices are geodesic simplices in
hyperbolic spaces:

\begin{defn}[geodesic simplex]
  Given two points $x$ and $y$ in $\mathbb{H}^n$, we define $[x,
    y]\colon [0,1] \to \Hyp^n$ to be the constant-speed
  parametrisation of the unique geodesic segment of $\mathbb{H}^n$
  joining $x$ to $y$. The standard $n$-simplex $\Delta^n$ is given by
  the following set: $\Delta^n = \{(x_0, \dots, x_n) \in \,
  \mathbb{R}^{n+1}_{\geq 0} \, | \, \sum_{i=0}^n x_i = 1 \}$. We
  identify $\Delta^{n-1}$ as the subset of $\Delta^n$ given by those
  points whose last coordinate is zero. A \emph{geodesic simplex}
  $\sigma \colon \Delta^n \rightarrow \mathbb{H}^n$ with vertices
  $x_0, \dots, x_n$, often denoted by $\straight(x_0, \cdots, x_n)$,
  is the map defined inductively as follows:
  $$
  \sigma\bigl((1-t) s + t(0, \dots, 0, 1)\bigr) \coloneqq [\sigma(s), x_n](t),
  $$
  where $s \in \Delta^{n-1}$ and $t \in [0,1]$. 

  We say that a singular chain in a hyperbolic $n$-manifold $M$ is
  \emph{geodesic} (or \emph{straight}) if each simplex with non-zero coefficient
  is the composition of a geodesic simplex with the
  universal covering projection $\mathbb{H}^n \rightarrow M$.
\end{defn}

\begin{rem}[geodesic simplices are Lipschitz]\label{rem:lipschitz:cost:diameter}
      All geodesic simplices~$\Delta^k \to \Hyp^n$ are smooth, by
      smoothness of the exponential map. Moreover, geodesic simplices
      are Lipschitz maps, with Lipschitz constant depending only on
      their diameter~\cite[Proposition~2.4 and Remark~2.5]{Loh-Sauer}.
      This shows that chains supported on geodesic simplices of
      uniformly bounded diameter are Lipschitz.
\end{rem}

\begin{rem}[geodesic straightening]\label{rem:efficient:geode:cycle}
  Every singular cycle in a hyperbolic manifold $M$ is
  canonically homologous to a geodesic one with (at most) the same
  $\ell^1$-norm~\cite[Lemma~C.4.3]{BePe}.
\end{rem}

\begin{defn}[double integration]
  Let $M$ be an oriented Riemannian $n$-manifold without boundary, let
  $\alpha$ be a standard $\Pi(M)$-space, and let $c = \sum_{\sigma \in
    \map(\Delta^n,M)} f_\sigma \cdot \sigma \in C_n^\lf(M; \linfz
  \alpha)$ be an $\alpha$-parametrised locally finite chain that is
  Lipschitz, supported on smooth simplices, and such that~$|c|_1<
  \infty$.  We define the \emph{double integration} of~$c$ over~$M$ by
  \[
  \langle \dvol_M, c\rangle
  \coloneqq \bigl\langle \dvol_M, I_\R(c) \bigr\rangle
  =
  \sum_{\sigma \in \map(\Delta^n,M)}
  \biggl(\int_{\alpha(\sigma[0])} f_\sigma\biggr)
  \cdot
  \langle \dvol_M,\sigma\rangle.
  \]
  The hypotheses on~$c$ ensure that the sum on the right-hand side converges
  absolutely.
\end{defn}

We will now come to the recognition of fundamental cycles through integration;
for simplicity, we restrict to the case of tame manifolds and ergodic parameter
spaces.

\begin{defn}[tame manifold]
  Let $M$ be a connected non-compact manifold without boundary. We say
  that $M$ is \emph{tame} if it is homeomorphic to the interior of a
  connected compact manifold with boundary.
\end{defn}

In order to prove a characterisation of $\alpha$-parametrised locally
finite fundamental cycles of a tame $n$-manifold $M$ in terms of the
double integration map, we need to compute the top-dimensional  
homology~$H_n^\lf(M; \linfz \alpha)$.

\begin{prop}[top locally finite homology of tame manifolds with local coefficients]\label{prop:top:dim:homology}
  Let $M$ be a tame oriented connected $n$-manifold and let $\alpha$
  be an ergodic standard $\Pi(M)$-space.  Then the map~$H_n^\lf(M;\Z)
  \longrightarrow H_n^\lf(M;\linfz \alpha)$ induced by the inclusion
  of constant functions is an isomorphism, and so
   \[
  H_n^\lf\left(M; \linfz \alpha\right) \cong \mathbb{Z},
  \]
  generated by the $\alpha$-parametrised locally finite fundamental class~$[M]_\lf^\alpha$.
\end{prop}

\begin{proof}
  Let $W$ be the closure of~$M$.  The topological collar
  theorem~\cite{brownflat} shows that~$M \cong W \cup_{\partial W}
  (\partial W \times [0, +\infty))$. Via this identification, we may
    consider the compact subspaces~$K_r \coloneqq W \cup_{\partial W}
    (\partial W \times [0,r])$ of~$M$; clearly, the family~$(K_r)_{r
      \in \N}$ is cofinal in the directed set of all compact subspaces
    of~$M$. Therefore, we can write the locally finite chain complex
    as the inverse limit
  \[
  C_n^\lf(M; \linfz \alpha)
  \cong \varprojlim_{r \to\infty} C_n^\lf\bigl(M, M \setminus K_r; \linfz \alpha\bigr). 
  \]
  
  It is easy to check that the directed system $(C_n^\lf(M, M
  \setminus K_r; \linfz \alpha)_{r\in \N}$ satisfies the
  Mittag-Leffler condition~\cite[Definition~3.5.6]{Weibel}. Moreover,
  for each~$r\in\N$ we have
  \begin{align*}
    H_{n+1}^\lf\bigl(M, M \setminus K_r; \linfz \alpha\bigr)
    & \cong H_{n+1}\bigl(K_r, \partial K_r;\linfz \alpha\bigr) \\
    & \cong H_{n+1}\bigl(W,\partial W; \linfz \alpha\bigr) \cong 0.
  \end{align*}
  Thus, the $\lim^1$-term in the short exact sequence computing the
  homology of a limit of chain complexes~\cite[Theorem~3.5.8]{Weibel}
  vanishes, and we see that the inclusions $(M, \emptyset)
  \hookrightarrow (M, M\backslash K_r)$ induce an isomorphism
  $$
  H_n^\lf(M; \linfz \alpha) \cong
  \varprojlim_{r \to\infty} H_n^\lf\bigl(M, M \setminus K_r; \linfz \alpha\bigr).
  $$
  
  Since for~$r\in \N$ the pairs~$(M, M \setminus K_r)$ are homotopy
  equivalent to~$(W, \partial W)$ in a compatible way (namely, by
  collapsing~$M\backslash W$ onto~$\partial W$), the inverse
  system~$(H_n^\lf(M,M\setminus K_r;\linfz \alpha))_{r \in \N}$ is
  isomorphic to the constant system~$H_n(W,\partial W;\linfz \alpha)$,
  and we have an isomorphism
  $$
  H_n^\lf\bigl(M; \linfz \alpha\bigr)
  \cong H_n\bigl(W, \partial W; \linfz \alpha\bigr)
  $$
  induced by the collapse map~$M\to W$.
  
  Repeating the argument in the setting of constant $\Z$-coefficients
  yields a similar isomorphism, reducing our claim to the proof that
  the lower map in the following commutative diagram is an
  isomorphism:
  
  \[\begin{tikzcd}
  	 H_n^\lf(M; \Z) \arrow[r] \arrow[d, "\cong"]& H_n^\lf\bigl(M; \linfz \alpha\bigr) \arrow[d, "\cong"]\\
  	H_n(W, \partial W; \Z)\arrow[r] & H_n\bigl(W, \partial W; \linfz \alpha\bigr)
  \end{tikzcd}.\]
  
  Using Poincaré duality (with local coefficients), this can be
  further translated into a question about $0$-th cohomology:
 
  \[\begin{tikzcd}
 H_n(W, \partial W; \Z)\arrow[r] & H_n\bigl(W, \partial W; \linfz \alpha\bigr)\\
 H^0(W;\Z) \arrow[u, "\cong", "\mathrm{PD}"'] \arrow[r]& H^0\bigl(W ; \linfz \alpha\bigr) \arrow[u, "\cong", "\mathrm{PD}"'] 
 \end{tikzcd}.\]
 
  Because $\alpha$ is ergodic, we know that for every~$x_0 \in W$ the
  fixed module~$\linfz{\alpha(x_0)}^{\pi_1(W,x_0)}$ consists only of
  the constant functions.  In other words, the inclusion of~$\Z$
  into~$\linfz{\alpha(x_0)}^{\pi_1(W,x_0)}$ as constant functions
  induces an isomorphism on $0$-th twisted cohomology, which
  translates to the lower map in the previous diagram being an
  isomorphism.
\end{proof}

Finally, we can prove our criterion for $\alpha$-parametrised locally
finite fundamental cycles of tame manifolds.

\begin{prop}[recognising fundamental cycles through integration]\label{prop:double:integration}
  Let $M$ be a tame Riemannian $n$-manifold and let $\alpha$ be an
  ergodic standard $\Pi(M)$-space. Let $c \in C_n^\lf(M; \linfz
  \alpha)$ be a smooth Lipschitz $\alpha$-parametrised locally finite
  cycle with $|c |_1 < +\infty$. Then, the following are equivalent:
\begin{enumerate}
\item The chain $c$ is an $\alpha$-parametrised locally finite fundamental cycle;
\item $\langle \dvol_M, c \rangle = \vol(M)$;
\item $\bigl|\langle \dvol_M, c \rangle - \vol(M)\bigr| < \vol(M)$.
\end{enumerate}
\end{prop}

\begin{proof}
	Because $I_\R$ is a norm non-increasing chain map, the
	$\ell^1$-norm of~$I_\R(c)$ is finite. Let $k \in \R$ be the
	number that satisfies
	\[ \bigl[ I_\R(c) \bigr] = k \cdot [M]^\R_\lf, 
	\]
	where $[M]^\R_\lf \in H_n^\lf(M;\R)$~is the real fundamental class of~$M$.
	Then~\cite[(proof of) Proposition~4.4]{Loh-Sauer}
	\[ \langle \dvol_M, c \rangle
	= \langle \dvol_M, I_\R(c) \rangle
	= k \cdot \vol(M).
	\]
	Moreover, by Proposition~\ref{prop:top:dim:homology}, the canonical
	homomorphism~$H_n^\lf(M;\Z) \longrightarrow H_n^\lf(M;\linfz
        \alpha)$ is an isomorphism and both modules are isomorphic to~$\Z$,
	and the canonical homomorphism~$H_n^\lf(M;\Z) \longrightarrow H_n^\lf(M;\R)$
	corresponds to the inclusion~$\Z \longrightarrow \R$.
	Therefore, we obtain that $k$ is an integer and that 
	\[ [c] = k \cdot [M]^\alpha_\lf \in H_n^\lf(M;\linfz \alpha).
	\]
	Now the equivalence of the three statements easily follows.
\end{proof}

\subsection{Gaining boundary control}\label{subsec:gaincontrol}

One benefit of studying the integral foliated locally finite
simplicial volume of tame manifolds is that it gives an upper bound on
relative integral foliated simplicial volume of their closure,
including boundary control. This fact will play a crucial role in
deriving the main theorem (Theorem~\ref{thm:main}) from the
finite-volume hyperbolic case.

\begin{prop}\label{prop:boundary:control:inequality}
  Let $W$ be an oriented compact connected $n$-manifold, let $M \coloneqq W^\circ$,
  and let $\alpha$ be a standard $\Pi(M)$-space. Then
  \[ \ifsv {W,\partial W}^\alpha_\partial \leq \ifsv M ^\alpha_\lf.
  \]
\end{prop}
\begin{proof}
  If $\ifsv M ^\alpha_\lf=\infty$, there is nothing to show.
 
  Otherwise, we prove that for every $\alpha$-parametrised locally
  finite fundamental cycle~$c$ of~$M$ with~$|c|_1 < \infty$, and for every
  $\varepsilon >0$, there exists an $\alpha$-parametrised
  relative fundamental cycle $c'(\varepsilon)$ of $W$ such that
  $$
  \bigl| c'(\varepsilon) \bigr|_1 \leq \lvert c \rvert_1
  \qand
  \bigl| \partial c'(\varepsilon) \bigr|_1 \leq \varepsilon.
  $$

  The proof follows the one for ordinary simplicial
  volume~\cite[Proposition~5.12]{Lothesis}.  Since $M$ is the interior
  of $W$, we can identify $M \cong W \cup_{\partial W} (\partial W
  \times [0, +\infty))$ as in the proof of Proposition~\ref{prop:top:dim:homology}.
  Then, we can consider the exhaustion of~$M$ by compact sets given by
  $K_r \coloneqq W \cup_{\partial W} (\partial W \times [0, r])$ for all $r
  \in \, \mathbb{N}$. Let us write $c_r \in \, C_n(M; \linfz
  \alpha)$ for the (finite) chain obtained from~$c$ by setting to~$0$
  the coefficients of the simplices that do not intersect~$K_r$.
  By construction, we have a
  monotone increasing sequence $(|c_r|_1)_{r \in \, \mathbb{N}}$
  such that
  $$
  \lim_{r \rightarrow +\infty} |c_r|_1 = |c|_1.
  $$
  Therefore, since $|c|_1 < \infty$, it follows that $\lim_{r \rightarrow +\infty} |c|_1- |c_r|_1 = 0$,
  and so for each $\varepsilon \in \R_{>0}$,
  taking~$R \in \mathbb{N}$ sufficiently large yields
  $$|c|_1 - |c_R|_1 \leq \frac{\varepsilon}{n+1}.$$
  Since we have $|c|_1 - |c_R|_1 = |c - c_R|_1$, we get the following
  estimate on the $\ell^1$-norm of the boundary of $c_R$:
  $$
  |\partial c_R|_1 = \bigl|\partial (c - c_R)\bigr|_1 \leq (n+1) \cdot |c - c_R|_1 \leq (n+1) \cdot \frac{\varepsilon}{n +1} = \varepsilon.
  $$

  We now claim that the desired $\alpha$-parametrised relative
  fundamental cycle~$c'(\varepsilon)$ of~$W$ can be obtained by taking
  the pushforward of~$c_R$ along the collapse map $M \rightarrow W$.
  Indeed, this new $\alpha$-parametrised chain~$c'(\varepsilon) \in
  C_n(W; \linfz \alpha)$ satisfies
  \[ \bigl| c'(\varepsilon) \bigr|_1 \leq |c_R|_1
     \qand
     \bigl|\partial c'(\varepsilon)\bigr|_1 \leq |\partial c_R|_1.
  \]
  Performing the construction of restriction and projection
  to a locally finite \emph{integral} fundamental cycle~$z$ of~$M$
  and to a corresponding boundary~$\partial b$ to~$c$ for large enough~$r$ 
  leads to a relative fundamental cycle~$z'$ of~$W$ \cite[Theorem~5.4 and
    Proposition~5.12]{Lothesis} that differs by a boundary~$\partial b'$
  from~$c'(\varepsilon)$. Hence, $c'(\varepsilon)$ is a relative
  $\alpha$-parametrised fundamental cycle of~$W$.~This finishes the proof.
\end{proof}

\section{A proportionality principle\\
  for finite-volume hyperbolic $3$-manifolds}\label{sec:hyplf}

In this section, we will show that the locally finite integral
foliated simplicial volume satisfies a proportionality principle for
complete finite-volume hyperbolic $3$-manifolds.

\begin{repthm}{thm:hyp3}
  Let $M$ be an oriented complete connected finite-volume hyperbolic
  $3$-manifold (without boundary). Then
  \[ \ifsvlf M 
   = \frac{\vol(M)}{v_3},
  \]
  where $v_3$ is the volume of any regular ideal tetrahedron in
  hyperbolic $3$-space.
\end{repthm}

As a corollary of this result, we will be able to compute the precise
value of the $\widehat{\pi_1(M)}$-parametrised integral simplicial
volume of all compact hyperbolic $3$-manifolds $M$ with toroidal
boundary (Section~\ref{sec:hyprel}).

The main step of the proof of the proportionality
Theorem~\ref{thm:hyp3} consists in the following result, whose proof
is discussed in detail in
Sections~\ref{subsec:smearing}--\ref{sec:proof:prop:princ}.

\begin{thm}\label{thm:pp}
  Let $n \in \N$, let $N$ be an oriented closed connected hyperbolic
  $n$-manifold and let $M$ be an oriented complete connected
  hyperbolic $n$-manifold of finite volume. Then
  \[    \frac{\isv N}{\vol(N)}
   \geq \frac{\ifsvlf M}{\vol(M)}. 
  \]
\end{thm}

Moreover, we need the following classical computations:

\begin{rem}[proportionality for classical simplicial volume]\label{rem:simvolhypclassical}
  If $M$ is an oriented closed connected hyperbolic $n$-manifold, then
  Gromov and Thurston established the
  proportionality~\cite{Grom82,Thurston}\cite[Theorem~C.4.2]{BePe}
  \[ \sv M = \frac{\vol(M)}{v_n}.
  \]
  Furthermore, this proportionality also holds for
  \begin{itemize}
  \item the locally finite simplicial volume in the finite-volume case and
  \item the relative simplicial volume (also with boundary control) in the
    compact case with toroidal boundary.
  \end{itemize}
  By now, many proofs of these extensions can be found in the literature --
  some of them deriving the relative case from the finite-volume case, and
  some of them deriving the finite-volume case from the relative case.
  A concise exposition is for instance given by Fujiwara and
  Manning~\cite[Appendix~A]{FM}.
\end{rem}

We are now ready to prove the proportionality Theorem~\ref{thm:hyp3}.

\begin{proof}[Proof of Theorem~\ref{thm:hyp3}]
We first prove that 
$$
\ifsvlf M \geq \frac{\vol(M)}{v_3}. 
$$
By Proposition~\ref{prop:integration:coeff:inequality:ord:foliated},
we have the following inequality between locally finite integral
foliated simplicial volume and the classical locally finite simplicial
volume:
$$
\ifsvlf M \geq \rVert M \lVert_\lf.
$$
Therefore, the proportionality principle for complete finite-volume hyperbolic
manifolds (Remark~\ref{rem:simvolhypclassical}) allows us to conclude that
$$
\ifsvlf M \geq \lVert M \rVert_{\lf} = \frac{\vol(M)}{v_3}.
$$

We now prove the opposite inequality. The stable integral simplicial
volume of every closed hyperbolic $3$-manifold $Z$ agrees with its
classical simplicial volume~\cite[Theorem~1.7]{FLPS}. This shows that
$$
\stisv Z = \Vert Z \rVert = \frac{\vol(Z)}{v_3},
$$
where the last equality comes from the proportionality principle for
closed hyperbolic manifolds (Remark~\ref{rem:simvolhypclassical}). The
crucial step of the proof is now Theorem~\ref{thm:pp}.  Indeed, since
the volume is multiplicative with respect to finite coverings, we have
$$
\frac{\stisv Z}{\vol(Z)} = \inf \Bigl\{\frac{\lVert N \rVert_\mathbb{Z}}{\vol(N)}
\Bigm| N \rightarrow Z \mbox{ is a finite covering} \Bigr\}. 
$$
Applying Theorem~\ref{thm:pp}, we can conclude that
$$
\frac{1}{v_3} = \frac{\stisv Z}{\vol(Z)} \geq \frac{\ifsvlf M}{\vol(M)}. 
$$
This finishes the proof.
\end{proof}

\begin{rem}
  It is worth noting that Theorem~\ref{thm:hyp3} cannot hold in higher
  dimensions.  Indeed, in the closed case, in dimension~$n \geq 4$,
  there is a uniform gap between integral foliated simplicial volume
  and simplicial volume for hyperbolic
  $n$-manifolds~\cite[Theorem~1.8]{FLPS}.
\end{rem}

\subsection{Setup for the construction of the smearing map}\label{subsec:smearing}

We now come to the preparations for the proof of Theorem~\ref{thm:pp}.
In order to prove this weak form of proportionality, we will stick to
a simplistic, discrete version of smearing (which will be
enough for our purposes), based on the following notion of meshes
and straightening.

\begin{defn}[decorated mesh of a Riemannian manifold]
        Let $M$~be a Riemannian manifold without boundary. A \emph{mesh} $\mathcal{P} =
        (P_i)_{i \in I}$ is a locally finite partition
        of~$M$ into (at most) countably many Borel sets obtained 
        from a locally finite triangulation of
        $M$. Moreover, we ask that the universal covering projection be
        trivial over each $P_i$.
        
        A \emph{decorated mesh} of $M$ is a pair
        $(\mathcal{P}, T)$, where $\mathcal{P}$ is a mesh of $M$ and
        $T =(t_i)_{i \in I}$ is a countable subset of
        $M$ such that $t_i \in \, P_i$ for each $i \in I$.
        
        We define the \emph{size} of a (decorated) mesh $\mathcal{P}$
        to be the supremum of all the diameters of the Borel sets
        appearing in $\mathcal{P}$.
\end{defn}

\begin{setup}\label{setup:smearing}
  Let $n \in \N$, let $N$ be an oriented \emph{closed} connected
  hyperbolic $n$-manifold (with fundamental group~$\Lambda$ and
  universal covering projection~$\pi_N \colon \Hyp^n \longrightarrow
  N$), and let $M$ be an oriented connected \emph{complete
    finite-volume} hyperbolic $n$-manifold (with fundamental
  group~$\Gamma \coloneqq \pi_1(M,x_0)$ and universal covering
  projection~$\pi_M \colon \Hyp^n \longrightarrow M$).  Moreover, let
  $G \coloneqq \Isom^+(\Hyp^n)$ be the group of orientation-preserving
  isometries of~$\Hyp^n$.
  
  We first choose an integral fundamental cycle~$c$ of~$N$, which by
  Remark~\ref{rem:efficient:geode:cycle} we may assume is geodesic.
  We denote by~$\Delta$ the maximal diameter of lifts to~$\Hyp^n$ of
  simplices with non-zero coefficient in~$c$.

  We construct meshes on~$M$ and~$\Hyp^n$ as follows:
  \begin{itemize}
  \item We choose a mesh~$\mathcal P$ of $M$ whose size is smaller than the
      global injectivity radius of~$N$.
  \item By the triviality condition in the definition of the mesh
    $\mathcal{P}$, we can choose, for each~$P_i$, a homeomorphic
    lift~$\tilde P_i \subset \Hyp^n$.  We
    write~$\widetilde{\mathcal{P}} \coloneqq (\widetilde{P_i} )_{i \in
      I}$.  The union~$D \coloneqq \bigcup_{i\in I}\tilde P_i$ is then
    a Borel fundamental domain for the deck transformation action
    of~$\Gamma$ on~$\Hyp^n$.  Then $\Gamma \cdot
    \widetilde{\mathcal{P}}$~is a mesh of~$\mathbb{H}^n$ with size
    smaller than the global injectivity radius of~$N$.
  \item We now impose one further restriction on the size of the Borel
    sets in~$\mathcal{P}$ (and~$\Gamma \cdot \tilde {\mathcal P}$),
    which can be attained through a suitable subdivision.  For each
    $i\in I$, let $K_i$~be the closure of the union of all Borel sets
    in~$\Gamma \cdot \tilde{\mathcal P}$ that intersect the
    $\Delta$-neighbourhood of~$\tilde{P_i}$.  Now consider the volume
    function
    \begin{align*}
       \vol \colon (\mathbb{H}^n)^{n+1} &\rightarrow \mathbb{R} \\
       (x_0, \dots, x_n) &\mapsto \vol_{\Hyp^n}\bigl(\straight(x_0, \dots, x_n)\bigr).
    \end{align*}
    Since $K_i^{n+1} \subset (\mathbb{H}^n)^{n+1}$ is compact, and
    $\vol$~is continuous with respect to the maximum metric~$d_\infty$
    on~$(\mathbb{H}^n)^{n+1}$, the restriction of~$\vol$
    to~$K_i^{n+1}$ is uniformly continuous.  In particular, there
    exists~$\delta_i >0$ such that whenever
    $\underline{x}$~and~$\underline{y}$ in~$K_i^{n + 1}$ satisfy
    $d_\infty(\underline{x}, \underline{y}) \le \delta_i$, we have
    \begin{equation}\label{eq.closevol}
      \bigl|\vol(\underline{x}) - \vol(\underline{y}) \bigr| <
      \frac{\vol(N)}{2 \cdot |c|_1} .
    \end{equation}
    
    The fact that $\Gamma \cdot \tilde{\mathcal{P}}$~is locally finite
    implies that for each~$j \in I$, the Borel set~$\tilde P_j$ is
    contained in at most finitely many $\Gamma$-translates of
    subsets~$K_i$. Taking $\varepsilon_j$~to be the (strictly
    positive) minimum among the respective~$\delta_i$, we refine the
    mesh~$\mathcal{P}$ by (locally finitely) subdividing each~$P_j$
    into Borel sets of diameter less than~$\varepsilon_j$.
  \end{itemize}
  We will use the following straightening procedure: 
  \begin{itemize}
  \item For~$k\in \N$, we write~$\widetilde S_k \subset
    \map(\Delta^k,\Hyp^n)$ for the set of geodesic simplices
    in~$\Hyp^n$ whose vertices all lie in~$\Gamma \cdot \widetilde T$
    and whose $0$-vertex lies in~$\widetilde T$. For the corresponding
    set of simplices on~$M$, we use the notation $S_k
    \coloneqq\{\pi_M \circ \tau \mid \tau \in \widetilde S_k
    \}$. For~$\varrho \in S_k$, we designate by~$\tilde \varrho$ the
    unique $\pi_M$-lift of~$\varrho$ in~$\widetilde S_k$.
  \item For~$k \in \N$, we consider the function
  \[\snap \colon \map(\Delta^k, \Hyp^n) \longrightarrow \Gamma \cdot \widetilde S_k\]
    that ``snaps'' the vertices of~$\sigma \colon \Delta^k
    \to \Hyp^n$ to the grid~$\Gamma \cdot \widetilde T$
    and then takes the geodesic simplex associated with this tuple of
    vertices. More precisely, let $Q_0, \dots, Q_k$ be the Borel sets in
    the mesh~${\Gamma \cdot \widetilde{P}}$ containing the vertices of
    $\sigma$, and let $q_0, \dots, q_k$ the corresponding points in
    $\Gamma \cdot \widetilde T$. Then, the map snap sends $\sigma$ to
    the geodesic simplex in $\Gamma \cdot \widetilde S_k$ spanned by the
    vertices~$q_0, \dots, q_k$.
  \end{itemize}
  We consider the standard $\Gamma$-space~$Z_M$ on~$M$ given by
  \begin{itemize}
  \item the probability space~$G/\Lambda$ (with
    the normalised Haar measure $\overline{\mu}$) and
  \item the $\Gamma$-action by
    left translation on~$G$.
  \end{itemize}
  Here, we endow $G$ with the (unimodular) Haar measure $\mu$
  satisfying the normalisation $\mu(X_M) = 1$, where $X_M \subset G$
  is a Borel fundamental domain of the right $\Lambda$-action on $G$.

  Moreover, we will use the following parameter space on~$M$: For each
  point~$p\in M$, we choose a path~$c_p$ in~$M$ from the
  basepoint~$x_0\in M$ to~$p$, with the property that each~$c_p$ lifts
  to a path in~$\Hyp^n$ having both endpoints in the same
  $\Gamma$-translate of the fundamental domain~$D$.  We then extend
  the standard $\Gamma$-space~$Z_M$ to a standard
  $\Pi(M)$-space~$\alpha_M$ using this choice of paths
  (Remark~\ref{rem:groupoidvsgroup}).  Then $\alpha_M$~is ergodic by
  the Moore ergodicity theorem~\cite[Theorem~4.10.2]{wittemorris}.  We
  will also make use of the associated normed local coefficient
  system~$\linfz{\alpha_M}$ on~$M$
  (Definition~\ref{def:ass:norm:loc:coeff:syst}).
\end{setup}

\subsection{A useful Borel map}

The following maps allow to smear simplices of~$N$ over~$\Hyp^n$ and
then to associate the correct ``weight'' to them on~$M$.

\begin{defn}[a useful Borel map]
  In the situation of Setup~\ref{setup:smearing}, for $k \in \,
  \mathbb{N}$, for every singular simplex $\sigma \colon \Delta^k
  \rightarrow N$ and for every simplex $\tau \in \, \Gamma \cdot
  \widetilde S_k$, we define the function $f_{\sigma, \tau} \in \,
  L^\infty(Z_M, \mathbb{Z})$ as follows: We choose a $\pi_N$-lift
  $\tilde{\sigma}$ of $\sigma$ and then, we set
  $$
   f_{\sigma, \tau}(x \Lambda) \coloneqq \# \bigl\{ \lambda \in \Lambda \mid \snap(x\lambda \tilde\sigma) = \tau\}.
  $$
\end{defn} 

\begin{lemma}
  In the situation of Setup~\ref{setup:smearing}, the map $f_{\sigma,
    \tau}$ is a well-defined essentially bounded Borel map~$Z_M \longrightarrow \Z$.
\end{lemma}
\begin{proof}
        First of all, we prove that the definition of~$f_{\sigma,
          \tau}$ does not depend on the choice of lift~$\tilde\sigma$
        of $\sigma$. Indeed, for a different lift~$\lambda \tilde
        \sigma$ (where $\lambda \in \Lambda$), we have a bijection
	\begin{align*}
	\{\mu  \in \Lambda \mid \snap(x\mu \tilde\sigma) = \tau\} &\to \{\mu  \in \Lambda \mid \snap(x\mu \lambda \tilde\sigma) = \tau\}\\
	\mu &\mapsto \mu \lambda^{-1}.
	\end{align*}
	Similarly, independence of the choice of~$x$ over, say,
        $x\lambda$ (with $\lambda \in \Lambda$) as a coset
        representative for~$x\Lambda$ follows from the bijection
	\begin{align*}
	\{\mu  \in \Lambda \mid \snap(x\mu \tilde\sigma) = \tau\} &\to \{\mu  \in \Lambda \mid \snap(x \lambda \mu \tilde\sigma) = \tau\}\\
	\mu &\mapsto \lambda^{-1} \mu.
	\end{align*}
	This shows that $f_{\sigma, \tau}$~is a well-defined function
        (to~$\N \cup \{\infty\}$).

	To see that the function~$f_{\sigma, \tau}$ lives
        in~$L^\infty(Z_M;\Z)$, we first show that it is essentially
        bounded. In fact, we prove more: it only takes values in the
        set~$\{0,1\}$. Indeed, given~$x\in G$, we consider the translated
        decorated mesh~$(x^{-1} \cdot \Gamma\cdot
        \widetilde{\mathcal{P}}, x^{-1} \cdot \Gamma\cdot
        \widetilde{T})$ of~$\Hyp^n$. The size of this mesh is still
        less than the global injectivity radius of~$N$. Therefore, any
        two distinct $\Lambda$-translates of a simplex in
        $\mathbb{H}^n$ cannot have the vertices in the same Borel sets
        of the mesh $x^{-1} \cdot \Gamma\cdot
        \widetilde{\mathcal{P}}$. This means that no two such
        translates can have the vertices close to the ones of
        $x^{-1}\tau$.
        Looking again at the original
        mesh~$\Gamma \cdot \widetilde{\mathcal{P}}$, this
        means that for each simplex~$\tilde \sigma \in \Hyp^n$, there
        is at most one~$\lambda \in \Lambda$ such that $\snap(x\lambda
        \tilde \sigma) = \tau$.
	
	To see that $f_{\sigma, \tau}$~is a Borel-measurable function,
        observe that by definition, we can express~$f_{\sigma,\tau}$ as
	\[x\Lambda \mapsto \begin{cases*}
	1 &\text{if there exists~$\lambda \in \Lambda$ with $\snap(x\lambda \tilde \sigma) = \tau$}\\
	0 & \text{otherwise}
	\end{cases*}.\]
	The function~$f_{\sigma, \tau}$ is thus the characteristic
        function of a subset~$A \subseteq G/\Lambda$, which
        we claim is Borel.
        
        To this end, we express~$A$
        differently: consider the (continuous, hence measurable) map
	\begin{align*}
		 \psi\colon G \times \Lambda \times (\Hyp^n)^{k+1} & \to (\Hyp^n)^{k+1}\\
		 (x, \lambda, \alpha) &\mapsto x\lambda\cdot \alpha,
	\end{align*}
	where $G$~acts diagonally on the right hand side. Notice
        that a tuple~$\alpha \in (\Hyp^n)^{k+1}$ uniquely determines a
        geodesic simplex in~$\Hyp^n$, and vice versa, so we now
        incorporate this identification into the notation. Let $Q_0,
        \ldots, Q_k$ be the Borel subsets in the mesh~$\Gamma
        \cdot (\tilde P_i)_{i\in I}$ containing, respectively, the
        vertices~$\tau[0], \ldots, \tau[k]$ of~$\tau$. Then it is easy
        to check that $A$~is the set obtained by projecting
	\[  (G\times \Lambda \times \{\tilde\sigma\})\cap \psi^{-1}(Q_0 \times \ldots \times Q_k)\]
	via the composition~$G\times \Lambda \times \{\tilde \sigma\}
        \to G\to G/\Lambda$.
        
        This construction of~$A$ starts
        with the intersection of a Borel set and a closed set, which
        is thus Borel. Hence, we are only left to show that
        the two projection maps that follow take Borel sets to Borel sets.
        Let us denote by $\iota_\lambda\colon G \to
        G\times \Lambda \times \{\tilde \sigma \}$ the inclusion
        of~$G$ at the~$(\lambda, \tilde \sigma)$-slice.  Since the
        projection~$G\times \Lambda \times \{\tilde \sigma\} \to G$
        takes each Borel subset~$B$ to the countable
        union~$\bigcup_{\lambda\in\Lambda} \iota_\lambda^{-1}(B)$, it
        sends Borel sets to Borel sets. Finally, if $B$~is a Borel
        subset of~$G$, then for the quotient map~$p\colon G \to
        G/\Lambda$ we see that~$p^{-1}(p(B)) = \bigcup_{\lambda \in
          \Lambda}\lambda \cdot B$ is Borel, so $p(B)$~is Borel.
\end{proof}


\subsection{The smearing map and its properties}\label{sec:smearing:map:def}

\begin{defn}[the parametrised discrete smearing map]
	In the situation of Setup~\ref{setup:smearing}, for~$k \in
        \N$, we define the \emph{parametrised discrete smearing map in
          degree~$k$} by
	\begin{align*}
	\varphi_k \colon C_k(N;\Z) & \longrightarrow C_k(M; L^\infty(\alpha_M, \mathbb{Z})) 
	\\
	\sigma	& \longmapsto \sum_{\varrho \in S_k} f_{\sigma, \tilde \varrho} \cdot \varrho.
	\end{align*}
\end{defn}

\begin{lemma}
	In the situation of Setup~\ref{setup:smearing}, we have:
	\begin{enumerate}
	        \item For each~$k \in \N$, the map~$\varphi_k$ is well-defined.
		\item The sequence~$(\varphi_k)_{k\in \N}$ is a chain map.
                \item The image of~$\varphi_*$ consists of smooth chains. 
	\end{enumerate}
\end{lemma}

\begin{proof}
  \emph{Ad~1.} 
 	We have to show that $\varphi_k$ sends integral finite chains
        to locally finite chains with local coefficients, and it is
        clearly enough to show that this is true for each
        simplex~$\sigma$ on~$N$.  So let $K\subset M$ be a compact
        subset and let $\rho \in S_k$ be a simplex with non-zero
        coefficient in~$\phi_k(\sigma)$ that intersects~$K$. We will
        see that these conditions on~$\rho$ restrict it to a finite
        number of possibilities.
        
        Since there is an isometric translate of~$\tilde \sigma$ 
        that snaps to~$\tilde\rho$, by the triangle inequality, we see that
        that $\tilde\rho$~must have diameter at most
        \[d \coloneqq \diam(\tilde\sigma) + 2 \cdot \size(\mathcal P),\]
        and because the universal covering projection map~$\pi_M$
        is distance-non-increasing, the same upper bound is valid for~$\diam(\rho)$.
        
        It follows that $\rho[0]$~is in the $d$-neighbourhood of~$K$,
        and as $\mathcal P$~is locally finite, there are only finitely
        many~$P_1,\ldots,P_f \in \mathcal P$ intersecting this
        neighbourhood. Hence, $\tilde \rho[0]$ lies in one
        among~$\tilde P_1,\ldots,\tilde P_f \in \tilde{\mathcal P}$.
        The image of~$\tilde \rho$ must therefore be contained in a
        $d$-neighbourhood of the union~$\bigcup_{j=1}^f \tilde
        P_j$. This is a bounded subset of~$\Hyp^n$, which again
        intersects only finitely many Borel sets~$Q_1, \ldots, Q_m$ of
        the locally finite mesh~$\Gamma \cdot \tilde{\mathcal P}$.
        Denoting by~$t_1, \ldots, t_m$ the corresponding points of the
        decorated structure~$\Gamma\cdot\tilde T$, we see that~$\tilde
        \rho$ is constrained to have its vertices in the set~$\{t_1,
        \ldots, t_m \}$.  As a $\tilde\rho$ is geodesic, and geodesic
        simplices are entirely determined by their (ordered) vertices,
        we conclude there are only finitely-many possibilities
        for~$\tilde\rho$, and thus the same is true of~$\rho$.

   \emph{Ad~2.}  We now prove that $\varphi_k$~is a chain map, by
        showing that it is compatible with each face
        map~$\partial_j$. In fact, we shall present only the
        computation of the most delicate case~$j=0$, the remaining
        ones being easily recovered by suppressing all occurrences of
        the correction term~$\gamma_\varrho$. By Definition~\ref{def:chainslc} we have,
        for each~$\sigma\colon \Delta^k \to N$ 
	\[\partial_0(\varphi_k(\sigma)) = \sum_{\varrho \in S_k} (f_{\sigma, \tilde \varrho}\cdot \gamma_\varrho) \cdot \partial_0 \varrho,\]
	where $\gamma_\varrho$~denotes the element~$c_{\varrho[0]} *
        \varrho[0,1]* c_{\varrho[1]}^{-1} \in \Gamma$.  
	It is straightforward to check that each~$\gamma \in \Gamma$
        acts on~$f_{\sigma, \tau}$ by $f_{\sigma, \tau} \cdot \gamma =
        f_{\sigma, \gamma^{-1} \tau}$, so the previous formula can be
        rewritten as
	\[ \partial_0(\varphi_k(\sigma)) =
        \sum_{\varrho \in S_k} f_{\sigma, \gamma_\varrho^{-1} \tilde \varrho}  \cdot \partial_0 \varrho.\]
	Expressing this summation in terms of the $(k-1)$-simplices yields:
	\[ \partial_0(\varphi_k(\sigma)) =
        \sum_{\tau \in S_{k-1}} \bigg( \sum_{\substack{\varrho \in S_k\\ \partial_0 \varrho = \tau}} f_{\sigma, \gamma_\varrho^{-1} \tilde \varrho}\bigg) \cdot \tau. \]
	
	We now claim that for each~$\tau \in S_{k-1}$, the inner
        summation simplifies as follows:
	\begin{equation}\label{eq.simplexsum}
	\sum_{\substack{\varrho \in S_k\\ \partial_0 \varrho = \tau}}f_{\sigma, \gamma_\varrho^{-1} \tilde\varrho} = f_{\partial_0\sigma, \tilde \tau}.
	\end{equation}
	After this step is justified, the proof will be complete, since
	\[\sum_{\tau \in S_{k-1}} f_{\partial_0\sigma , \tilde \tau}  \cdot \tau = \varphi_{k-1} (\partial_0 \sigma).\]
	
	We now justify Formula~(\ref{eq.simplexsum}). The function on
        the left-hand side assigns to each coset~$x\Lambda\in
        G/\Lambda$ the value
	$\sum_{\substack{\varrho \in S_k\\ \partial_0 \varrho = \tau}} \# \{\lambda \in \Lambda \mid \snap(x\lambda \tilde\sigma) = \gamma_\varrho^{-1} \tilde \varrho \}$,
	which is the cardinality of the disjoint union
	\[ A\coloneqq \coprod_{\substack{\varrho \in S_k\\ \partial_0 \varrho = \tau}} \{\lambda \in \Lambda \mid \snap(x\lambda \tilde\sigma) = \gamma_\varrho^{-1} \tilde\varrho \}. \]
	
	On the other hand, the right-hand side of
        Formula~(\ref{eq.simplexsum}) assigns to~$x\Lambda$ the cardinality of
	\begin{align*}
	B\coloneqq& \{\lambda \in \Lambda \mid \snap(x\lambda\partial_0 \tilde\sigma) = \tilde\tau\} \\
	= &\{\lambda \in \Lambda \mid \partial_0 \snap(x\lambda \tilde\sigma) = \tilde\tau\}.
	\end{align*}

	 Since each of~$c_{\varrho[0]}, c_{\varrho[1]}$ lifts to a
         path with both endpoints in the same $\Gamma$-translate
         of~$D$, and for each $\varrho\in S_k$ the simplex $\tilde
         \varrho$~has its $0$-vertex in~$D$, it follows that
         $\gamma_\varrho^{-1}$~acts by bringing the $1$-vertex
         of~$\tilde\varrho$ to~$D$. In other words,
         if~$\partial_0\varrho = \tau$, then~$\partial_0(
         \gamma_\varrho^{-1} \tilde \varrho) = \tilde\tau$.
         
   	 Having made this observation, it is straightforward to check
         that the map from $A$~to~$B$ taking $(\lambda, \varrho)\mapsto
         \lambda$ is a well-defined bijection, with inverse $\lambda
         \mapsto (\lambda ,\pi_M\circ\snap(x\lambda\tilde\sigma))$.

   \emph{Ad~3.}  As the chains in the image of~$\varphi_*$ consist
     of geodesic simplices, they are also smooth
     (Remark~\ref{rem:lipschitz:cost:diameter}).
\end{proof}

\begin{prop}\label{prop.fundclass}
	In the situation of Setup~\ref{setup:smearing}, the
        map~$H_n(\varphi_*)$ sends the integral fundamental class~$[N]$
        of~$N$ to the $\alpha_M$-parametrised locally finite
        fundamental class~$[M]_\lf^{\alpha_M}$ of~$M$.
\end{prop}

In order to prove Proposition~\ref{prop.fundclass}, we first establish
an upper bound for the norm of the chain map~$\varphi$, whence also
for the norm of~$H_n(\varphi_*)$.

\begin{lemma}\label{lemma:norm:varphi}
	In the situation of Setup~\ref{setup:smearing}, we have for all~$k \in \N$:
	\begin{enumerate}
		\item If $\sigma \in \map(\Delta^k, N)$, then
		\[ \sum_{\varrho \in S_k} \int_{G/\Lambda} f_{\sigma, \tilde\varrho} \d\overline\mu
		= \sum_{\varrho \in S_k} \int_{G/\Lambda} |f_{\sigma,\tilde\varrho}| \d\overline\mu
		= \frac{\vol(M)}{\vol(N)}.
		\]
		\item In particular, $\| \varphi_k\| \leq \vol(M) / \vol(N)$.
	\end{enumerate}
\end{lemma}

\begin{proof}
        Since we are dealing with the $\ell^1$-norms on the chain
        complexes, it suffices to prove the first part. Let $\sigma
        \in \map(\Delta^k,N)$ and let $\widetilde \sigma$ be a
        $\pi_N$-lift of~$\sigma$. We set~$t_0 \coloneqq
        \widetilde\sigma[0] \in \Hyp^n$ and $X_N \coloneqq \{ x \in G
        \mid x(t_0) \in D\}$. Then $X_N$ is a Borel fundamental domain
        for the left action of~$\Gamma$ on~$G$ and (because the Haar
        measure on~$G$ pushes forward, under the canonical
        projection~$G \longrightarrow \Hyp^n$, to a scalar multiple of
        the hyperbolic volume)
	\[ \mu(X_N) = \frac{\mu(X_N)}{\mu(X_M)}
                    = \frac{\vol(M)}{\vol(N)}.
	\]
	
	Moreover, we have
	\begin{align*}
	  \bigl| \varphi_k(\sigma)\bigr|_1 &= \sum_{\varrho \in S_k}\int_{G/\Lambda}|f_{\sigma, \tilde\varrho}| \d \overline\mu
          \\&
	= \int_{G/\Lambda} \bigl( x\Lambda\mapsto \#\{\lambda \in \Lambda \mid x\lambda\tilde\sigma[0] \in D \} \bigr)\d \overline\mu\\
	&=  \int_{X_M} \bigl( x\mapsto \#\{\lambda \in \Lambda \mid x\lambda \in X_N \} \bigr)\d\mu
	\\&
        = \int_{X_M} \;\sum_{\lambda \in \Lambda} \chi_{X_N \cdot \lambda^{-1}} \d\mu\\
	&= \sum_{\lambda \in \Lambda} \mu(X_M \cap X_N\cdot \lambda^{-1})
        \\&
	 = \sum_{\lambda \in \Lambda} \mu(X_M \cdot \lambda \cap X_N)\\
	&= \mu(X_N) = \frac{\vol(M)}{\vol(N)}. 
	\end{align*}
	
	Because the functions~$f_{\sigma,\tilde\varrho}$ all are
        non-negative (by construction), we also have $\sum_{\varrho
          \in S_k} \int_{G/\Lambda} f_{\sigma,\tilde\varrho}
        \d\overline\mu = | \varphi_k(\sigma) |_1 = \vol(M)/\vol(N)$.
        This finishes the proof.
\end{proof}

\begin{proof}[Proof of Proposition~\ref{prop.fundclass}]
        We will prove that $H_n(\varphi_*)([N])$ is an
        $\alpha$-para\-metrised locally finite fundamental class via
        double integration. Let $c$ be the geodesic fundamental cycle
        of $N$ appearing in Setup~\ref{setup:smearing}.  We aim to
        show that $\varphi_n(c)$ satisfies condition~(3) of
        Proposition~\ref{prop:double:integration}. However, in order
        to apply this proposition, we first have to check that
        $\varphi_n(c)$ is Lip\-schitz and of finite $\ell^1$-norm.
        Since $\varphi_n$ has bounded norm by
        Lemma~\ref{lemma:norm:varphi} and $|c |_1 < +\infty$, also
        $\varphi_n(c)$ has finite $\ell^1$-norm. In order to show that
        $\varphi_n(c)$ is Lipschitz, let us call $L$ the size of the
        mesh~$\mathcal{P}$. Then, all pairs of vertices of the
        geodesic simplices lying in $\tilde{S}_n$ that contribute
        to~$\varphi_n(c)$ have distance at most~$\Delta +
        2L$. Therefore, they all have diameter bounded by a constant
        depending only on $\Delta$ and~$L$. Therefore, by
        Remark~\ref{rem:lipschitz:cost:diameter}, we know that also
        their Lipschitz constants are all bounded by a uniform
        constant depending only on $\Delta$ and~$L$. This condition
        clearly reflects to the simplices of~$S_n$. This proves that
        $\varphi_n(c)$ is Lipschitz.
	
        We are now ready to apply double integration to
        $\varphi_n(c)$. Writing~$c=\sum_{j=1}^{t}a_j \sigma_j$, we
        have, on the one hand,
        \[ \bigl\langle \dvol_M, I_\R(\varphi_n(c)) \bigr\rangle
        = \sum_{j = 1}^t a_j \cdot \biggl( \sum_{\varrho \in \, S_n} \biggl(\int_{G/\Lambda} f_{\sigma_j, \tilde{\varrho}}  \ d\overline{\mu} \biggr) \cdot \int_{\Delta^n} \varrho^* \dvol_M \biggr),
        \]
        where $\overline{\mu}$ denotes the normalised Haar measure
        described in Setup~\ref{setup:smearing}; on the other hand,
        because $c$ is a fundamental cycle of~$N$, we have
        \begin{align*}
          \vol(M) &= \frac{\vol(M)}{\vol(N)} \cdot
          \biggl( \sum_{j = 1}^t a_j \int_{\Delta^n} \tilde{\sigma}_j^* \dvol_{\mathbb{H}^n} \biggr) .
        \end{align*}
        Therefore, using Lemma~\ref{lemma:norm:varphi}.1, we obtain: 
        \begin{align*}
          &\bigl| \langle \dvol_M, I_\mathbb{R}(\varphi_n(c)) \rangle - \vol(M) \bigr| 
          \\
          & = \biggl|\sum_{j = 1}^t a_j \sum_{\varrho \in \, S_n}
          \biggr(\int_{G/\Lambda} f_{\sigma_j, \tilde{\varrho}}  \ d\overline{\mu} \biggr)
          \cdot \biggr(\int_{\Delta^n} \varrho^* \dvol_M - \int_{\Delta^n} \tilde{\sigma}_j^* \dvol_{\mathbb{H}^n} \biggr)
          \biggr|
          \\
          & = \biggl|\sum_{j = 1}^t a_j \sum_{\varrho \in \, S_n}
          \biggl(\int_{G/\Lambda} f_{\sigma_j, \tilde{\varrho}}  \ d\overline{\mu} \biggr)
          \cdot \biggr(\int_{\Delta^n} \tilde{\varrho}^* \dvol_{\mathbb{H}^n} - \int_{\Delta^n} \tilde{\sigma}_j^* \dvol_{\mathbb{H}^n} \biggr)\biggr|
          \\
          & \leq \sum_{j = 1}^t |a_j|
          \sum_{\varrho \in \, S_n} \biggl( \int_{G/\Lambda} |f_{\sigma_j, \tilde{\varrho}}|  \ d\overline{\mu} \biggr)
          \cdot \biggl|\int_{\Delta^n} \tilde{\varrho}^* \dvol_{\mathbb{H}^n} - \int_{\Delta^n} \tilde{\sigma}_j^* \dvol_{\mathbb{H}^n} \biggr|.
          \end{align*}
          
          Now observe that whenever $\tilde \sigma_j$~is not a
          $G$-translate of some simplex that snaps to~$\tilde\rho$,
          the function~$f_{\sigma_j, \tilde \rho}$ is identically~$0$,
          and so the summands corresponding to such~$j, \varrho$
          vanish.  For the remaining summands~$j, \varrho$, we see
          that the volume estimate given by
          Formula~(\ref{eq.closevol}) from Setup~\ref{setup:smearing}
          applies to~$\tilde \sigma_j,\tilde\rho$.  The previous
          expression is thus bounded above by
          \begin{align*}
          &\sum_{j = 1}^t |a_j|  \sum_{\varrho \in \, S_n}
          \biggr( \int_{G/\Lambda} |f_{\sigma_j, \tilde{\varrho}}|  \ d\overline{\mu} \biggr)
          \cdot \frac{\vol(N)}{2 \cdot |c |_1} 
          \\
          & \leq \biggl( \sum_{j = 1}^t |a_j| \biggr)
          \cdot \frac{\vol(M)}{\vol(N)} \cdot \frac{\vol(N)}{2 \cdot |c|_1} = \frac{\vol(M)}{2} .
        \end{align*}
       	
       	We have thus verified condition (3) from Proposition~\ref{prop:double:integration},
       	which finishes the proof.       
\end{proof}

\subsection{Proof of Theorem~\ref{thm:pp}}\label{sec:proof:prop:princ}

The proof of Theorem~\ref{thm:pp} now follows easily from the previous construction.

\begin{proof}[Proof of Theorem~\ref{thm:pp}]
  Let $n \in \, \mathbb{N}$. The hypotheses on $N$ and $M$ allow us to
  assume the situation of Setup~\ref{setup:smearing}. Thus, we have
  the chain map
  $$
  \varphi_* \colon C_*(N; \mathbb{Z}) \rightarrow C_*(M; L^\infty(\alpha_M, \mathbb{Z}))
  $$
  described in Section~\ref{sec:smearing:map:def}. Let $c \in \,
  C_n(N; \mathbb{Z})$ be an integral fundamental cycle of $N$. As
  proved in Propositions~\ref{prop.fundclass}
  and~\ref{lemma:norm:varphi}~(2), we have that $\varphi_n(c) \in \,
  C_n(M; L^\infty(\alpha_M, \mathbb{Z}))$ is an
  $\alpha_M$-parametrised locally finite fundamental cycle whose
  $\ell^1$-norm is bounded as follows:
  $$
  |\varphi_n(c) |_1 \leq \frac{\vol(M)}{\vol(N)} \cdot |c|_1 .
  $$
  By taking the infimum on both sides we get the claim
  \[    \frac{\isv N}{\vol(N)}
  \geq \frac{\ifsvlf M}{\vol(M)}.
  \qedhere 
  \]
\end{proof}

\section{A proportionality principle\\
for hyperbolic $3$-manifolds with toroidal boundary}\label{sec:hyprel}

From Theorem~\ref{thm:hyp3}, we can now derive the corresponding
relative result with boundary control:

\begin{repcor}{cor:hyp3rel}
  Let $W$ be an oriented compact connected hyperbolic $3$-mani\-fold
  with empty or toroidal boundary and let $M\coloneqq W^\circ$. Then 
  \[ \ifsv{W,\partial W}^{\pfc{\Pi(W)}}_\partial
   = \frac{\vol(M)}{v_3}.
  \]
\end{repcor}

As in the corresponding result for the closed case, the proof will use
input from ergodic theory, more specifically of the ergodic theoretic
properties of the profinite completion of fundamental groups of
hyperbolic $3$-manifolds (Section~\ref{subsec:pfchyp}).

\subsection{The profinite completion of hyperbolic $3$-manifold groups}\label{subsec:pfchyp}

The dependency of parametrised simplicial volume on the chosen
parameter space seems to be a difficult problem (similar to
the fixed price problem for cost of groups). In some special
cases, it is known that the profinite completion of the fundamental
group is the ``best'' parameter space. This can be formalised
in terms of \emph{weak containment} of parameter spaces
and the fact that parametrised simplicial volume is monotone with
respect to weak containment of parameter spaces. These technical
definitions and statements have been deferred to
Appendix~\ref{subsec:weakcont}.

\begin{defn}[Property~$\EMD*$]\label{def.emd*}
  An infinite countable group~$\Gamma$ \emph{satisfies~$\EMD*$} if
  every ergodic standard $\Gamma$-space is weakly contained in the
  profinite completion~$\widehat \Gamma$ of~$\Gamma$.
\end{defn}

In the following we prove that hyperbolic 3-manifolds with empty or
toroidal boundary satisfy~$\EMD*$.  In the closed case, this has
been already noticed by Kechris~\cite[p.~487]{kechris} and Bowen and
Tucker-Drob~\cite[p.~212]{bowen} and proved in detail by Frigerio,
L\"oh, Pagliantini and Sauer~\cite[Proposition~3.10]{FLPS}.

\begin{prop}\label{prop:emd}
  Let~$W$ be an oriented compact connected hyperbolic $3$-man\-i\-fold
  with empty or toroidal boundary. Then~$\pi_1(W)$ satisfies~$\EMD*$.
\end{prop}
\begin{proof}
  By Agol's virtual fibre theorem~\cite{agol,friedlkitayama}, there is
  a finite covering~$N$ of~$W$ that fibers over~$S^1$.  Then,
  $\pi_1(N)$ is a semidirect product of~$\Lambda$ and~$\Z$,
  where~$\Lambda$ is the fundamental group of an oriented compact
  surface.  For residually finite groups, the property~$\EMD*$ is
  equivalent to the property~$\MD$~\cite[Theorem~1.4]{td}, which is a
  property related to profinite completions introduced by
  Kechris~\cite[p.~464]{kechris}.  Since surface groups
  satisfy~$\MD$~\cite[Theorem~1.4]{bowen} and also free groups
  satisfy~$\MD$~\cite[Theorem~1]{kechris}, also~$\Lambda$
  satisfies~$\MD$.  By taking~$\Lambda$ as normal subgroup
  of~$\pi_1(N)$, we apply $\MD$-inheritance~\cite[Theorem~1.4]{bowen}
  to see that~$\pi_1(N)$ satisfies~$\MD$ and as a residually finite
  group also satisfies~$\EMD*$.  For residually finite groups the
  property~$\EMD*$ is preserved under passing from a finite index
  subgroup to the ambient group and therefore, $\pi_1(W)$
  satisfies~$\EMD*$.
\end{proof}
  
\subsection{Proof of Corollary~\ref{cor:hyp3rel}}\label{sec:proof:cor:prof}

After these preparations, we can derive Corollary~\ref{cor:hyp3rel} from Theorem~\ref{thm:hyp3}:

\begin{proof}[Proof of Corollary~\ref{cor:hyp3rel}]
  On the one hand, we have (Proposition~\ref{prop:boundaryreal} and
  Remark~\ref{rem:simvolhypclassical})
  \[ \ifsv{W,\partial W}^{\pfc{\pi_1(W)}}_\partial
     \geq \sv{W,\partial W} = \frac{\vol(M)}{v_3}.
  \]

  On the other hand, we can argue as follows: The proportionality principle
  for finite-volume hyperbolic $3$-manifolds (Theorem~\ref{thm:hyp3}) and
  Proposition~\ref{prop:boundary:control:inequality} show that
  \[ \ifsv{W,\partial W}_\partial
     \leq \ifsv {M} _\lf
     = \frac{\vol(M)}{v_3}.
  \]
  Moreover, $\pi_1(W)$ satisfies~$\EMD*$ (Proposition~\ref{prop:emd})
  and $\ifsv {W, \partial W}_\partial$ can be computed via ergodic
  spaces (Proposition~\ref{prop:ergodicsuff}). Therefore, monotonicity
  of parametrised simplicial volume with respect to weak containment
  (Proposition~\ref{prop:weakcont}) shows that
  \[ \ifsv{W,\partial W}^{\pfc{\pi_1(W)}}_\partial
     \leq \ifsv{W,\partial W}_\partial
     \leq \frac{\vol(M)}{v_3}.
     \qedhere
  \]
\end{proof}

\subsection{Summary of the relative case}

We conclude this section by showing that all the integral foliated
variations of the simplicial volume of compact hyperbolic
$3$-manifolds with toroidal boundary agree. Moreover, they will also
provide the same value as the ordinary simplicial volume and the
\emph{ideal simplicial volume}. Recall that the ideal simplicial
volume $\lVert \cdot \rVert_\mathcal{I}$ is a homotopy invariant of
compact manifolds~\cite{FMo:ideal}. It differs from the standard simplicial
volume by allowing the presence of \emph{ideal} simplices in
representatives of the fundamental class.

\begin{cor}\label{cor:all:simpl:vol:equal}
  Let $W$ be an oriented compact connected hyperbolic $3$-mani\-fold
  with empty or toroidal boundary. If $M$ denotes the the interior of
  $W$, then all the following versions of relative simplicial volume
  equal~$\frac{\vol(M)}{v_3}$:
  $$
  \ifsv{W, \partial W}_\partial^{\pfc{\pi_1(W)}} , \, \ifsv{W, \partial W}_\partial ,\,  \ifsv{W, \partial W} ,\, \lVert W, \partial W \rVert  ,\, \lVert W, \partial W \rVert_\mathcal{I}.
  $$
\end{cor}
\begin{proof}
  The interior~$M$ is a complete finite-volume hyperbolic
  manifold. Recall by Proposition~\ref{prop:boundaryreal} and
  Remark~\ref{rem:simvolhypclassical} that
  $$
  \ifsv{W, \partial W} \geq \lVert W, \partial W \rVert = \frac{\vol(M)}{v_3} .
  $$
  Thus, applying Proposition~\ref{prop:boundary:control:inequality}
  and Theorem~\ref{thm:hyp3}, it follows that
  $$
  \frac{\vol(M)}{v_3} = \ifsvlf M \geq \ifsv{W, \partial W}_\partial \geq \ifsv{W, \partial W} \ge \frac{\vol(M)}{v_3} .
  $$

  The proportionality principle for~$\ifsv{W, \partial
    W}_\partial^{\pfc{\pi_1(W)}}$ is the content of
  Corollary~\ref{cor:hyp3rel}. For the ideal simplicial volume~$\lVert
  W, \partial W \rVert_\mathcal{I}$, it has been shown before by
  Frigerio and the third author~\cite{FMo:ideal}.
\end{proof}

\section{The glueing step}\label{sec:glue}

Our computation of parametrised simplicial volume of
complete finite-volume hyperbolic $3$-manifolds includes boundary
control (Corollary~\ref{cor:hyp3rel}). This allows us to prove the
following upper bound.

\begin{thm}\label{thm:jsjupperbound}
  Let $M$~be an oriented compact connected $3$-manifold with empty or
  toroidal boundary. If $M$~is prime and not covered by~$S^3$, then
  \[ \stisv{M,\partial M}
     = \ifsv{M,\partial M}^{\pfc{\Pi(M)}}
     \leq \frac{\hypvol(M)}{v_3}.
  \]
\end{thm}

The proof is based on the JSJ decomposition. We recall
terminology and notation around the JSJ decomposition in
Section~\ref{subsec:jsj}. The fundamental glueing estimates
are established in Section~\ref{subsec:glue}, specifics of
the $3$-manifold situation are discussed in
Section~\ref{subsec:profin3}, and the proof of
Theorem~\ref{thm:jsjupperbound} is given in
Section~\ref{subsec:proofjsj}.
Finally, we use Theorem~\ref{thm:jsjupperbound} to prove our main
Theorem~\ref{thm:main} in Section~\ref{subsec:proofmainthm}.

\subsection{The JSJ decomposition}\label{subsec:jsj}

In this section, we state one of the key ingredients allowing us to
assemble Corollary~\ref{cor:hyp3rel} (which pertains to
\emph{hyperbolic} $3$-manifolds of finite volume) and the
computations for \emph{Seifert-fibered} manifolds not covered
by~$S^3$~\cite[Section~8]{LP}\cite{FFL}
into a statement about \emph{prime} manifolds not covered
by~$S^3$. Notice that along this section we will only work with
irreducible manifolds. Indeed, our argument does not apply to~$S^1
\times S^2$. Howeover, it is known that $S^1 \times S^2$ satisfies
the integral approximation~\cite{Sthesis}\cite{LP}.

Most of the results presented here are by now classical, and can be
found, for example, in Aschenbrenner-Friedl-Wilton's
compendium~\cite{afw} and Martelli's book~\cite{Martelli_GT}.

We first remind the reader of some standard terminology: 
\emph{Seifert-fibered} manifolds are compact $3$-manifolds that admit
a certain type of decomposition into disjoint circles~\cite{afw}.
An oriented compact connected $3$-manifold~$M$ is 
\emph{atoroidal} if every $\pi_1$-injective continuous map~$S^1 \times S^1 \to M$
can be homotoped into~$\partial M$.

We can now state the JSJ decomposition Theorem.

\begin{thm}[JSJ decomposition~\protect{\cite[Theorem~1.6.1]{afw}}]
	Let $M$~be an oriented compact connected $3$-manifold with
        empty or toroidal boundary. If $M$~is irreducible, then there
        is a (possibly empty) collection of disjointly embedded
        tori~$T_1, \ldots, T_m \subset M$, such that each piece
        obtained by cutting~$M$ along~$\bigcup_{i} T_i$ is atoroidal
        or Seifert-fibered. Up to isotopy, there is a unique such
        collection with minimal number of tori.
\end{thm}

In keeping with tradition, we have stated the JSJ decomposition
Theorem for irreducible manifolds, although it also applies to~$S^1
\times S^2$, the only prime manifold (without spherical boundary
components) that is not irreducible. Indeed, $S^1 \times S^2$~is
atoroidal and Seifert-fibered.

The atoroidal pieces in the JSJ decomposition Theorem are, as we will
now explain, suited to the methods developed throughout
Section~\ref{sec:hyprel}. Indeed, the Hyperbolisation Theorem
\cite[Theorem~1.7.5]{afw} ensures that every piece that is not
Seifert-fibered is either hyperbolic, or has finite fundamental
group. In our situation, we can however rule out the latter
possibility, because every piece with finite fundamental group would
have to be closed (and in particular, the only piece in the JSJ
decomposition) and hence, by the Elliptisation Theorem
\cite[Theorem~1.7.3]{afw}, covered by~$S^3$. We are excluding such
manifolds from our main results by hypothesis.

\begin{defn}[$\hypvol$]\label{def:hypvol}
	Let $M$ be an oriented compact connected $3$-manifold with
        empty or toroidal boundary. If $M$ is irreducible, we denote
        by~$\hypvol(M)$ the sum of the volumes of the hyperbolic
        pieces in the JSJ decomposition of~$M$. We extend this
        definition to prime manifolds by setting $\hypvol(S^1 \times
        S^2)\coloneqq0$.
\end{defn}

\subsection{Basic glueing estimates}\label{subsec:glue}

As in the case of vanishing parametrised simplicial
volume~\cite[Propositions~4.4 and~4.5]{FFL}, we can prove the
following glueing estimates for glueings along tori. At this
point it is essential that we have control over the boundary
contributions. 

\begin{prop}[glueing estimate]\label{prop:gluetorus}
	Let $W$~be an oriented compact connected $n$-manifold with
        $n\ge2$, and let $\alpha$~be an essentially free standard
        $\Pi(W)$-space. Let $T\subset W$ be an embedded $(n-1)$-torus
        that separates $W$ into two pieces~$W_1, W_2$.  For each $i
        \in \{1,2\}$, assume the inclusion $T \hookrightarrow W_i$ as
        a boundary-component is $\pi_1$-injective and denote
        by~$\alpha_i$ the restriction of~$\alpha$ to~$W_i$.  Then
	  \[ \ifsv {W, \partial W}^{\alpha}_\partial
	     \leq \ifsv {W_1,\partial W_1}^{\alpha_1}_\partial
	        + \ifsv {W_2,\partial W_2}^{\alpha_2}_\partial.
	  \]
\end{prop}

\begin{proof}
  We proceed as in the proof with vanishing parametrised simplicial
  volume~\cite[proof of Proposition~4.4]{FFL}.  Let $\varepsilon,
  \varepsilon_\partial>0$. For each~$i \in \{1,2\}$, let $c_i \in
  C_n(W_i;\linfz{\alpha_i})$ be a relative fundamental cycle
  of~$(W_i,\partial W_i)$ satisfying
  \[|c_i|_1 \leq \ifsv{W_i,\partial W_i}^{\alpha_i}_\partial + \varepsilon
      \qand |\partial c_i|_1 \leq \varepsilon_\partial.\]

  Denoting by~$\alpha_0$ the restriction of~$\alpha$ to~$T$, we see
  that $c_0 \coloneqq (\partial c_1)|_{T} + (\partial c_2)|_{T}$ is a
  null-homologous cycle in~$C_{n-1}(T;\linfz{\alpha_0})$ with $|c_0|_1
  \leq |\partial c_1|_1 + |\partial c_2|_1 \leq 2 \cdot
  \varepsilon_\partial$.
  
  As $\alpha_0$~is essentially free, there is an $(n-1)$-UBC constant
  $K >0$ for $C_*(T;
  \linfz{\alpha_0})$~\cite[Proposition~4.1]{FFL}\cite[Theorem~1.3]{fauserloehUBC}.
  This means there is a chain~$b \in C_n(T;\linfz{\alpha_0})$ with
  \[ \partial b = c_0
     \qand |b|_1 \leq K \cdot |c_0|_1 \leq K \cdot 2 \cdot \varepsilon_\partial.
  \]
  Then the local criterion shows that
  $c \coloneqq c_1 + c_2 - b$ is a relative $\alpha$-parametrised
  fundamental cycle of~$(W,\partial W)$~\cite[Proposition~3.13]{FFL}. Moreover,
  \begin{align*}
    |c|_1
    & \leq \ifsv{W_1,\partial W_1}^{\alpha_1}_\partial + \ifsv{W_2,\partial W_2}^{\alpha_2}_\partial + 2 \cdot \varepsilon + K \cdot 2 \cdot \varepsilon_\partial
    \qand
    \\
    |\partial c|_1
    & \leq |\partial c_1|_1 + |\partial c_2|_1
    \leq 2 \cdot \varepsilon_\partial.
  \end{align*}
  Taking first $\varepsilon \to 0$ and then~$\varepsilon_\partial \to 0$,
  proves the claim.
\end{proof}

\begin{prop}[self-glueing estimate]\label{prop:selfgluetorus}
  Let $W$ be an oriented compact connected manifold of dimension~$n
  \geq 2$, let $T_1, T_2 \subseteq \partial W$ be two different
  $\pi_1$-injective components of~$\partial W$ that are homeomorphic
  to a torus, and let $f \colon T_1 \longrightarrow T_2$ be an
  orientation-reversing homeomorphism. We consider the glued
  manifold~$W' \coloneqq W / (T_1 \sim_f T_2)$ and an essentially free
  standard $\Pi(W)$-space~$\alpha$ as well as the induced standard
  $\Pi(W')$-space~$\alpha'$ on~$W'$. Then
  \[ \ifsv{W',\partial W'}^{\alpha'}_\partial
     \leq \ifsv{W,\partial W}^\alpha_\partial.
  \]
\end{prop}
\begin{proof}
  We can argue in the same way as in the proof of
  Proposition~\ref{prop:gluetorus}.
\end{proof}

\subsection{Profinite completions in dimension~$3$}\label{subsec:profin3}

Since the glueing results of the previous section involve restrictions
of parameter spaces, we will need to understand the effect of
restriction on the parameter spaces associated with the profinite
completion.

\begin{prop}\label{prop:jsjrestriction}
  Let $M$ be an irreducible oriented compact connected $3$-manifold
  with empty or toroidal boundary, and let $W$ be a piece of the JSJ
  de\-com\-po\-si\-tion of~$M$.
  \begin{enumerate}
  \item If $W$~is Seifert-fibered and not covered by~$S^3$, then
    \[ \ifsv{W,\partial W}^{\pfc{\Pi(M)}}_\partial = 0.\]
  \item If $W$ is hyperbolic, then
    \[ \ifsv{W,\partial W}^{\pfc{\Pi(M)}}_\partial
       \leq \ifsv{W,\partial W}^{\pfc{\Pi(W)}}_\partial.
    \]
  \end{enumerate}
  Here, the occurrences of~$\pfc{\Pi(M)}$ are
  to be interpreted as the restrictions of these standard
  $\Pi(M)$-spaces to~$\Pi(W)$.
\end{prop}

The proof relies on the following two facts:

\begin{lemma}\label{lem.efficient}
	Let $M$~be an irreducible oriented compact connected
        $3$-manifold with empty or toroidal boundary, let $W$~be a
        piece of the JSJ decomposition of~$M$, and choose a
        basepoint~$x_0$ for~$W$. Then the map~$\pfc{\pi_1(W, x_0)} \to
        \pfc{\pi_1(M, x_0)}$ induced by the inclusion~$W\to M$
        embeds~$\pfc{\pi_1(W, x_0)}$ as a closed subgroup
        of~$\pfc{\pi_1(M, x_0)}$.
\end{lemma}
\begin{proof}
  This statement is contained in the stronger fact that the profinite
  topology on~$\pi_1(M)$ is \emph{efficient} with respect to the
  graph-of-groups decomposition induced by the JSJ decomposition. We
  will not make further use of these technical notions, so we direct the
  reader to the original paper of Wilton and Zalesskii for the precise
  definitions and proofs~\cite{wilton2010profinite}.
\end{proof}
  
\begin{lemma}[{\cite[Example~12]{gheysens2017fixed}}]\label{lem.amplification}
	Let $G$~be a locally compact second-countable group and $H\le
        G$~a closed subgroup, both equipped with the left Haar
        measures~$\mu_G, \mu_H$, respectively.  Then, as an
        $H$-probability space, $G$~is isomorphic to the product of~$H$
        with a measured space carrying a trivial $H$-action.
\end{lemma}
\begin{proof}
  In the language of Gheysens and Monod~\cite[Example~12]{gheysens2017fixed},
  this lemma is expressed as the
  statement that $G$~is an \emph{amplification} of~$H$.
\end{proof}
  
We will also make use of the notion of \emph{weak containment} of
standard $G$-spaces (for $G$~a countable group or a groupoid with
countable automorphism groups), and its relationship to integral
foliated simplicial volume (with boundary control), as explained
in Appendix~\ref{subsec:weakcont}. The main result is 
Proposition~\ref{prop:weakcont}, which may be treated as a
black box during the proof of Proposition~\ref{prop:jsjrestriction}.

\begin{proof}[Proof of Proposition~\ref{prop:jsjrestriction}]
	Whether we are in situation (1) or~(2),
        Lemma~\ref{lem.efficient} ensures that, for any choice of
        basepoint (which we now suppress from the notation),
        $\pfc{\pi_1(W)}$~sits as a closed subgroup
        of~$\pfc{\pi_1(M)}$. Applying Lemma~\ref{lem.amplification}
        and restricting along the canonical map~$\pi_1(W) \to
        \pfc{\pi_1(W)}$ yields an isomorphism of standard
        $\pi_1(W)$-spaces
	\[\pfc{\pi_1(M)} \cong \pfc{\pi_1(W)} \times \alpha, \]
	where $\alpha$~is some standard $\pi_1(W)$-space with trivial action.
		
	It is now easy to see from Definition~\ref{def.weakcont} that this implies we
        have a weak containment of standard $\pi_1(W)$-spaces
        $\pfc{\pi_1(W)} \prec \pfc{\pi_1(M)}$, which extends to the
        level of groupoids:
	\[\pfc{\Pi(W)} \prec \pfc{\Pi(M)}.\]
	
	Applying now Proposition~\ref{prop:weakcont} immediately
        yields~(2), and reduces~(1) to the proof that Seifert-fibered
        spaces~$W$ that are not covered by~$S^3$ satisfy
        $\ifsv{W,\partial W}^{\pfc{\Pi(W)}}_\partial = 0$. Such
        manifolds are encompassed by earlier work~\cite[Section~8]{LP}\cite{FFL},
        whence $\stisv{W,\partial
          W}=0$. Now, Proposition~\ref{prop:ifsvprofin} tells us that
        $\ifsv{W,\partial W}^{\pfc{\Pi(W)}} = \stisv{W,\partial W}$,
        and by Lemma~\ref{lem:boundaryvanishing} this vanishing transfers
        to~$\ifsv{W,\partial W}_\partial^{\pfc{\Pi(W)}}$, finishing the proof.
\end{proof}

\subsection{Proof of Theorem~\ref{thm:jsjupperbound}}\label{subsec:proofjsj}

We only need to combine our previous considerations. If $M$ satisfies
the hypotheses in Theorem~\ref{thm:jsjupperbound} and is irreducible,
we have
\begin{align*}
  \stisv {M,\partial M}
  & = \ifsv {M,\partial M}^{\pfc{\Pi(M)}}
  & \text{(Proposition~\ref{prop:ifsvprofin})}
  \\
  & \leq \ifsv {M,\partial M}^{\pfc{\Pi(M)}}_\partial
  \\
  & \leq \sum_{\text{$W$ JSJ piece of~$M$}} \ifsv{W,\partial W}_\partial^{\pfc{\Pi(M)}}
  & \text{(Propositions~\ref{prop:gluetorus}, \ref{prop:selfgluetorus})}
  \\
  & \leq \sum_{\text{$W$ hyperbolic piece of~$M$}} \ifsv{W,\partial W}_\partial^{\pfc{\Pi(W)}}
  & \text{(Proposition~\ref{prop:jsjrestriction})}
  \\
  & = \sum_{\text{$W$ hyperbolic piece of~$M$}} \frac{\vol(W^\circ)}{v_3}
  & \text{(Corollary~\ref{cor:hyp3rel})}
  \\
  & = \frac{\hypvol(M)}{v_3}.
\end{align*}
Propositions \ref{prop:gluetorus} and~\ref{prop:selfgluetorus} can
be applied, because fundamental groups of compact $3$-manifolds are
residually finite~\cite{hempel} (whence the action on the profinite
completion is essentially free) and the JSJ~pieces have $\pi_1$-injective
boundary consisting of tori.

The additional case $M\cong S^1 \times S^2$ can, for instance, be
treated via self-coverings of~$S^1$ and
Proposition~\ref{prop:ifsvprofin}~\cite{LP,FFL,Sthesis}.

\subsection{Proof of Theorem~\ref{thm:main}}\label{subsec:proofmainthm}

Finally, we can prove the main theorem, Theorem~\ref{thm:main}: 
On the one hand, it is well known (Section~\ref{sec:intro:soma}) that 
\[ \stisv {M,\partial M} \geq \sv{M,\partial M} = \frac{\hypvol(M)}{v_3}.
\]
On the other hand, Theorem~\ref{thm:jsjupperbound} gives us the
converse estimate
\[ \stisv {M,\partial M} \leq \frac{\hypvol(M)}{v_3},
\]
which finishes the proof.

\section{Proofs of the non-approximation results}\label{sec:nonapprox}

We use the first $L^2$-Betti number to establish the
non-ap\-prox\-i\-ma\-tion results stated in the introduction.

\begin{repthm}{thm-noapprox}
  Let $d \in \N_{\geq 3}$, let $m,n \in \N$, let $M_1,\dots, M_m$, $N_1,\dots, N_n$
  be oriented closed connected $d$-manifolds with the following properties:
  \begin{enumerate}
  \item We have $\sv {M_j} > 0$  for all~$j \in
    \{1,\dots, m\}$ as well as $\sv {N_k} = 0$ for all~$k \in \{1,\dots, n\}$.
  \item Moreover, $m + n - 1 - \sum_{k=1}^n 1 / |\pi_1(N_k)| > \sum_{j=1}^m \sv {M_j}$ (with
  the convention that $1/\infty := 0$).
  \end{enumerate}
  Then the connected sum~$M \coloneqq M_1 \connsum \dots \connsum M_m \connsum N_1
  \connsum \dots \connsum N_n$ does \emph{not} satisfy integral approximation
  for simplicial volume, i.e., we have~$\sv M < \stisv M$.
\end{repthm}
\begin{proof}
  In dimension~$d \geq 3$, simplicial volume is additive under
  connected sums~\cite[p.~10]{Grom82}\cite{BBFIPP}. Therefore,
  \[ \sv M = \sum_{j=1}^m \sv{M_j} + \sum_{k=1}^n \sv{N_k}
           = \sum_{j=1}^m \sv{M_j}.
  \]         
  On the other hand, we know that the first $L^2$-Betti number of~$M$
  satisfies~$\ltb 1 (M) \leq \stisv M$~\cite[Corollary~5.6]{Sthesis}
  (the same proof in fact also gives the improved constant~$1$). Therefore,
  it suffices to show that $\ltb 1 (M) > \sum_{j=1}^m \sv {M_j}$.
  The connected sum formula for $L^2$-Betti
  numbers~\cite[Theorem~1.35]{lueckl2} yields
  \begin{align*}
        \ltb 1 (M)
      = m + n -1
      \;+\; & \sum_{j=1}^m \bigl( \ltb 1 (M_j) - \ltb 0 (M_j)\bigr)
      \\
      \;+\; & \sum_{k=1}^n \bigl( \ltb 1 (N_k) - \ltb 0 (N_k)\bigr).
  \end{align*}
  For connected manifolds~$X$, we have~$\ltb 0 (X) = 1 / |\pi_1(X)|$. 
  Because $\sv {M_j} > 0$ and $d >0$, the fundamental group~$\pi_1(M_j)$ is
  infinite~\cite[p.~39f]{Grom82}. Therefore,
  \begin{align*}
        \ltb 1 (M)
    & \geq m + n -1
      + \sum_{j=1}^m \bigl( \ltb 1 (M_j) - 0\bigr)        
      + \sum_{k=1}^n \Bigl( \ltb 1 (N_k) - \frac1{|\pi_1(N_k)|}\Bigr)
      \\
    & \geq m + n - 1
      - \sum_{k=1}^n \frac1{|\pi_1(N_k)|}
      \\
    & > \sum_{j=1}^m \sv{M_j}.
  \end{align*}
  Therefore, $\sv M < \ltb 1 (M) \leq \stisv M$, as claimed.
\end{proof}

\begin{repcor}{cor:intro:nonapprox}
  Let $N$ be an oriented closed connected hyperbolic $3$-manifold and
  let $k > \vol(N)/v_3$. Then the oriented closed connected
  $3$-manifold~$M \coloneqq N \connsum \connsum^k (S^1)^3$
  satisfies~$\sv M < \stisv M$.
\end{repcor}
\begin{proof}
  We only need to verify that the hypotheses of
  Theorem~\ref{thm-noapprox} are satisfied: We know $\sv N =
  \vol(N)/v_3$ (Remark~\ref{rem:simvolhypclassical}) and $\sv{(S^1)^3}
  = 0$~\cite[p.~8]{Grom82}. Moreover,~$\pi_1((S^1)^3) \cong \Z^3$ is
  infinite, so~$1/|\pi_1((S^1)^3)| = 0$.
\end{proof}

\begin{repcor}{cor:0noapprox}
  Let $M$ be an oriented closed connected $3$-manifold with $\sv M = 0$. Then
  the following are equivalent:
  \begin{enumerate}
  \item The simplicial volume of~$M$ satisfies integral approximation, i.e., $\stisv M = \sv M$.
  \item The manifold~$M$ is prime and has infinite fundamental group
    or $M$ is homeomorphic to~$\R P^3 \connsum \R P^3$.
  \end{enumerate}	
\end{repcor}
\begin{proof}
  \emph{Ad~$1 \Longrightarrow 2$.}
  For the contraposition, we consider the case that $M = N_1 \connsum
  \dots \connsum N_n$ is a non-trivial prime decomposition of~$M$,
  i.e., $n \geq 2$ and none of the~$N_k$ is homeomorphic to~$S^3$.
  Then $0 = \sv M = \sum_{k=1}^n \sv {N_k}$~\cite[p.~10]{Grom82}\cite{BBFIPP}
  and so $\sv {N_k} =0$ for all~$k \in \{1,\dots,n\}$.

  Because of~$N_k \not \cong S^3$, we have~$|\pi_1(N_k)| \geq 2$ by the
  Poincaré Conjecture~\cite[Corollary 1.7.4]{afw}.
  We now distinguish the following cases:
  \begin{itemize}
  \item If $n > 2$, then
    \[ n - 1 - \sum_{k=1}^n \frac1{|\pi_1(N_k)|}
    \geq n - 1 - \frac n2
    = \frac n2 - 1
    > 0.
    \]
    Therefore, from Theorem~\ref{thm-noapprox}, we obtain~$\sv M < \stisv M$.
  \item
    If $n = 2$ and $|\pi_1(N_1)| > 2$ or~$|\pi_1(N_2)| > 2$, then again 
    \[ n - 1 - \sum_{k=1}^n \frac1{|\pi_1(N_k)|}
    > 2 - 1 - 1
    = 0 
    \]
    and, by Theorem~\ref{thm-noapprox}, $\sv M < \stisv M$.
  \item
    If $n = 2$ and $|\pi_1(N_1)| = 2$ and $|\pi_1(N_2)| = 2$, then
    $\pi_1(N_1) \cong \Z/2 \cong \pi_1(N_2)$.  By the Elliptisation
    Theorem \cite[Theorem~1.7.3]{afw}, $N_1$~and~$N_2$ are both
    spherical and thus homeomorphic to the quotient of~$S^3$ by a
    subgroup of $\mathrm{SO}(4)$ of order~$2$, which must be~$\{\pm
    \mathrm{Id}\}$, and hence $N_1$ and $N_2$ are homeomorphic to~$\R
    P^3$.
  \end{itemize}

  \emph{Ad~$2 \Longrightarrow 1$.}
  If $M$ is prime with infinite fundamental group and~$\sv M = 0$,
  then $M$ must be a graph manifold (with infinite fundamental
  group)~\cite{Soma}.  Therefore, we obtain~\cite{FFL}
  \[ \sv M = \stisv M.
  \]
  Moreover, also~$\R P^3 \connsum\R P^3$ satisfies~$\sv {\R P^3
    \connsum \R P^3} = 0 = \stisv{\R P^3 \connsum \R P^3}$ (because this
  manifold admits a non-trivial
  self-covering)~\cite[Section~8]{LP}.
\end{proof}

\appendix

\section{Weak containment}\label{subsec:weakcont}

Parametrised simplicial volume of closed manifolds satisfies monotonicity
with respect to weak containment of parameter spaces~\cite[Theorem~3.3]{FLPS}.
This property admits a straightforward generalisation to the relative case
(including boundary control).

\begin{prop}[monotonicity with boundary control]\label{prop:weakcont}
  Let $W$ be an oriented compact connected manifold (with possibly non-empty 
  boundary) and infinite fundamental group. Let $\alpha$ and $\beta$
  be essentially free standard $\Pi(W)$-spaces with~$\alpha \prec \beta$.
  Then
  \[ \ifsv {W,\partial W}^\beta_\partial
     \leq
     \ifsv {W,\partial W}^\alpha_\partial.
  \]
\end{prop}

For the sake of completeness, we carry out the transformation from the closed
case to the relative case in detail.

Let us first recall basics on weak containment. Roughly speaking, $\alpha \prec \beta$
means that every finite relation between Borel sets and groupoid morphisms in~$\alpha$
can be simulated in~$\beta$ with arbitrary precision.

\begin{defn}[weak containment]\label{def.weakcont}
 \hfil
 \begin{itemize}
 \item Let $\Gamma$ be a countable group and let $\alpha \colon \Gamma
   \actson (X,\mu)$ and $\beta \colon \Gamma \actson (Y,\nu)$ be
   standard $\Gamma$-spaces. Then $\alpha$ is \emph{weakly contained}
   in~$\beta$ (in symbols: $\alpha \prec \beta$) if the following
   holds: For all~$\varepsilon \in \R_{>0}$, all~$m \in \N$, all Borel
   sets~$A_1, \dots, A_m \subset X$, and all finite subsets~$F \subset
   \Gamma$, there exist Borel sets~$B_1, \dots, B_m \subset Y$ such
   that
   \[ \fa{j \in \{1,\dots,m\}} \fa{g \in F}
      \bigl| \mu(g^\alpha(A_j) \cap A_j) - \nu(g^\beta(B_j) \cap B_j)\bigr|
      < \varepsilon.
   \]
 \item Let $G$ be a connected groupoid with countable automorphism
   groups and let $\alpha$ and $\beta$ be standard $G$-spaces. Then
   $\alpha$ is \emph{weakly contained} in~$\beta$ (in symbols: $\alpha
   \prec \beta$) if the following holds: For one (whence every)
   object~$x_0$ and $\Gamma \coloneqq \Aut_G x_0$, the standard
   $\Gamma$-space~$\alpha(x_0)$ is weakly contained in the
   $\Gamma$-space~$\beta(x_0)$.
 \end{itemize}	
\end{defn}

\begin{rem}[an alternative characterization of weak containment]\label{rem:weakcontweak}
  Let $G$ be a connected groupoid with countable automorphism groups
  and let $\alpha$ and $\beta$ be standard $G$-spaces. Moreover, let
  $x_0$ be an object of~$G$ and $\Gamma \coloneqq \Aut_G x_0$ as well
  as~$(X,\mu) \coloneqq \alpha(x_0)$.  Then $\alpha \prec \beta$ if
  and only if $\alpha(x_0)$ lies in the closure of
  \[ \bigl\{ \xi \in A(\Gamma,X,\mu)
     \bigm| \xi \cong_\Gamma \beta(x_0)
     \bigr\}
  \]
  with respect to the weak topology on the space~$A(\Gamma, X,\mu)$ of
  all $\mu$-preserving Borel actions of~$\Gamma$
  on~$(X,\mu)$~\cite[Proposition~10.1]{kechrisglobal}.
\end{rem}

\begin{proof}[Proof of Proposition~\ref{prop:weakcont}]
  The proof is a straightforward adaption of the proof in the closed
  case~\cite[Theorem~3.3]{FLPS}; we only need to convert the proof
  from twisted to local coefficients and add boundary control.

  Let $x_0 \in W$ and $\Gamma \coloneqq \pi_1(W,x_0)$.  Without loss
  of generality, we may assume that $\alpha$ is induced from a
  standard $\Gamma$-space~$\alpha_0 \colon \Gamma \actson (X,\mu)$
  at~$x_0$ and that $\beta$ is induced from a standard
  $\Gamma$-space~$\beta_0$ at~$x_0$
  (Remark~\ref{rem:groupoidvsgroup}).

  Let $n \coloneqq \dim W$, let $\varepsilon_\partial, \varepsilon \in
  \R_{>0}$ and let $c\in C_n(W;\linfz \alpha)$ be an
  $\alpha$-parametrised relative fundamental cycle with
  \[ \mathopen|\partial c|_1^\alpha \leq \frac12 \cdot \varepsilon_\partial.
  \]
  It then clearly suffices to show the following claim:
  \begin{itemize}
  \item[(C)] There exists a standard $\Pi(W)$-space~$\xi$ with~$\xi \cong \beta$
    and a $\xi$-pa\-ram\-e\-trised relative fundamental cycle~$c' \in C_n(W;\linfz \xi)$ 
    with
    \[ \mathopen| c' |_1^\xi \leq \mathopen|c|_1^\alpha + \varepsilon
       \qand
       \mathopen| \partial c' |^{\res \xi}_1 \leq \varepsilon_\partial.
    \]
  \end{itemize}
  
  To establish claim~(C), let $z \in C_n(W;\Z)$ be a relative
  fundamental cycle.  Because $c$ is a relative fundamental cycle,
  there exist chains~$b \in C_{n+1}(W;\linfz \alpha)$ and $w \in
  C_n(\partial W; \linfz{\res \alpha})$ with
  \[ c = z + \partial b + w \quad \text{in~$C_n(W;\linfz \alpha)$};
  \]
  more explicitly, we write
  \[ b =  \sum_{\tau \in T} f_\tau \cdot \tau
     \qand
     w = \sum_{\varrho \in R} g_{\varrho} \cdot \varrho 
  \]
  with finite sets~$T \subset \map(\Delta^{n+1}, W)$, $R \subset
  \map(\Delta^n,\partial W)$ and bounded measurable
  functions~$(f_\tau)_{\tau \in T}$ and~$(g_{\varrho})_{\varrho \in
    R}$ on~$(X,\mu)$.

  We choose a finite Borel partition~$X = A_1 \sqcup \dots \sqcup A_m$ of~$X$
  that refines the finite set
  \[ \bigl\{ f_\tau^{-1}(k)
     \bigm| \tau \in T,\ k \in \Z
     \bigr\}
     \cup
     \bigl\{ g_\varrho^{-1}(k)
     \bigm| \varrho \in R,\ k \in \Z
     \bigr\}.
  \]
  Moreover, we set
  \begin{align*}
    \delta & \coloneqq \frac1m \cdot
    \min \Bigl(\frac{\varepsilon}{\sum_{\tau \in T} \|f_\tau\|_\infty},\
               \frac{\varepsilon_\partial}{2 \cdot \sum_{\varrho \in R} \|g_\varrho\|_\infty}
    \Bigr),
    \\
    F & \coloneqq \bigl\{ \tau[0,1] \bigm| \tau \in T \bigr\}
    \cup \bigl\{ \varrho[0,1]\bigm| \varrho \in R \bigr\},
    \\
    F_0 & \coloneqq \{ h_f^{-1} \mid f \in F \} \subset \Gamma
  \end{align*}
  (the construction of the~$h_f$ is explained in Remark~\ref{rem:groupoidvsgroup}).
  Because $\alpha$ is weakly contained in~$\beta$, there exists a standard
  $\Gamma$-space~$\xi_0 \in A(\Gamma,X,\mu)$ such that
  \[ \fa{j \in \{1,\dots, m\}} \fa{h \in F_0^{-1}}
     \mu \bigl( h^{\alpha_0}(A_j) \symmdiff h^{\xi_0}(A_j) \bigr) < \delta
  \]
  and~$\xi_0 \cong_\Gamma \beta_0$ (Remark~\ref{rem:weakcontweak}).
  Let $\xi$ be the standard $\Pi(W)$-space associated with~$\xi_0$
  (Remark~\ref{rem:groupoidvsgroup}) and let
  \[ c' \coloneqq z + \partial b + w \quad \text{in~$C_*(W;\linfz \xi)$}
  . 
  \]
  By construction, $c'$ is a relative $\xi$-parametrised fundamental
  cycle of~$(W,\partial W)$ and $\xi \cong \beta$.

  We now show that $c'$ satisfies the norm estimates postulated in~(C).
  We have
  \begin{align*}
    \bigl| |c|_1^\alpha - |c'|_1^\xi
    \bigr|
    & = \biggl| \biggl| z + \sum_{j=1}^{n+1} \sum_{\tau \in T} (-1)^j \cdot f_\tau \cdot \partial_j \tau
                          + \sum_{\tau \in T} \alpha(\tau[0,1])(f_\tau) \cdot \partial_0 \tau
                          + w \biggr|_1^\alpha
    \\
    &\quad   -  \biggl| z + \sum_{j=1}^{n+1} \sum_{\tau \in T} (-1)^j \cdot f_\tau \cdot \partial_j \tau
                          + \sum_{\tau \in T} \xi(\tau[0,1])(f_\tau) \cdot \partial_0 \tau
                          + w \biggr|_1^\xi
        \biggr|
    \\
    & \leq \biggl|  z + \sum_{j=1}^{n+1} \sum_{\tau \in T} (-1)^j \cdot f_\tau \cdot \partial_j \tau
                          + \sum_{\tau \in T} \alpha(\tau[0,1])(f_\tau) \cdot \partial_0 \tau
                          + w 
    \\
    &\quad   -  \biggl( z + \sum_{j=1}^{n+1} \sum_{\tau \in T} (-1)^j \cdot f_\tau \cdot \partial_j \tau
                          + \sum_{\tau \in T} \xi(\tau[0,1])(f_\tau) \cdot \partial_0 \tau
                          + w \biggr)
                \biggr|_1^{(X,\mu)}
    \\
    & \leq \sum_{\tau \in T} \bigl\| \alpha(\tau[0,1])(f_\tau) - \xi(\tau[0,1])(f_\tau)
                          \bigr\|_1.
  \end{align*}
  For each~$\tau \in T$, we write~$f_\tau = \sum_{j=1}^m a_{\tau,j} \cdot \chi_{A_j}$
  with~$a_1, \dots, a_m \in \Z$. Then
  \begin{align*}
    \bigl\| \alpha(\tau[0,1])(f_\tau) - \xi(\tau[0,1])(f_\tau)
    \bigr\|_1
    & \leq \sum_{j=1}^m |a_j| \cdot \mu\bigl( h^{-1}_{\tau[0,1]}{}^{\alpha_0} (A_j)
                                 \symmdiff  h^{-1}_{\tau[0,1]}{}^{\xi_0} (A_j)\bigr)
    \\
    & \leq m \cdot \|f_\tau\|_\infty \cdot \delta.                             
  \end{align*}
  Hence,
  \[ \bigl| |c|_1^\alpha - |c'|_1^\xi \bigr|
  \leq \sum_{\tau \in T} m \cdot \|f_\tau\|_\infty \cdot \delta
  \leq \varepsilon.
  \]

  Similarly, we can handle $\partial c$~and~$\partial c'$. We have
  \begin{align*}
    \partial c
    & = \partial z + \partial \partial b + \partial w
    \quad \text{in~$C_{n-1}(\partial W;\linfz {\res \alpha})$}
    \\
    \partial c'
    & = \partial z + \partial \partial b + \partial w
    \quad \text{in~$C_{n-1}(\partial W;\linfz {\res \xi})$}.
  \end{align*}
  Moreover, $\partial \partial b =0$ in both cases and $\partial z$
  does not depend on the parameter space (as $z$ has constant coefficients).
  The same type of calculations as above shows that (where
  we simplify notation by writing~$\xi$ instead of~$\res \xi$) 
  \begin{align*}
    \bigl| |\partial c|_1^\alpha - |\partial c'|_1^\xi \bigr|
    & \leq \sum_{\varrho \in R} \bigl\| \alpha(\varrho[0,1])(g_\varrho)
    - \xi(\varrho[0,1])(g_\varrho) \bigr\|_1
    \\
    & \leq \sum_{\varrho \in R} m \cdot \|g_\varrho\|_\infty \cdot \delta
    \\
    & \leq \frac12 \cdot \varepsilon_\partial.
  \end{align*}
  Hence,
  \[ |\partial c'|_1^\xi
  \leq |\partial c|_1^\alpha + \frac12 \cdot \varepsilon_\partial
  \leq \varepsilon_\partial.
  \]
  This finishes the proof of claim~(C). 
\end{proof}

\newpage
 
\bibliographystyle{amsalphaabbrv}
\bibliography{biblionote}

\end{document}